\documentclass[reqno,12pt]{amsart}
\usepackage{a4wide}
\usepackage{amsmath}
\usepackage{amssymb}
\usepackage{amsthm}

\numberwithin{equation}{section}

\newtheorem{theo}{Theorem}

\newtheorem{coro}{Corollary}
\newtheorem{prop}{Proposition}

\newtheorem{lem}{Lemma}

\theoremstyle{remark}
\newtheorem{Remark}{Remark}
\newtheorem{Remarks}[Remark]{Remarks}

\def\al{\alpha}

\def\ep{\varepsilon}

\def\om{\omega}

\def\Th{\Theta}

\def\Om{\Omega}
\def\({\left(}
\def\){\right)}
\def\[{\left[}
\def\]{\right]}

\def\fl#1{\left\lfloor#1\right\rfloor}

\def\lcm{\operatorname{lcm}}

\begin{document}

\title[]{On the integrality of the Taylor coefficients of mirror maps, II}

\author[]{C. Krattenthaler$^\dagger$ and T. Rivoal}
\date{\today}

\address{C. Krattenthaler, Fakult\"at f\"ur Mathematik, Universit\"at Wien,
Nordbergstra{\ss}e~15, A-1090 Vienna, Austria.
WWW: \tt http://www.mat.univie.ac.at/\~{}kratt.}

\address{T. Rivoal,
Institut Fourier,
CNRS UMR 5582, Universit{\'e} Grenoble 1,
100 rue des Maths, BP~74,
38402 Saint-Martin d'H{\`e}res cedex,
France.\newline
WWW: \tt http://www-fourier.ujf-grenoble.fr/\~{}rivoal.}

\thanks{$^\dagger$Research partially supported 
by the Austrian Science Foundation FWF, grants Z130-N13 and grant S9607-N13,
the latter in the framework of the National Research Network 
``Analytic Combinatorics and Probabilistic Number Theory''}



\begin{abstract}
We continue our study begun in 
{\em ``On the integrality of the Taylor coefficients of mirror maps''} 
[Duke Math. J. (to appear)] 
of  the fine integrality properties of the Taylor coefficients of the series
${\bf q}(z)=z\exp({\bf G}(z)/{\bf F}(z))$, where ${\bf F}(z)$ and 
${\bf G}(z)+\log(z) {\bf F}(z)$ 
are specific solutions of certain hypergeometric differential
equations with maximal unipotent monodromy at $z=0$. More precisely, 
we address the question of finding the largest integer $v$
such that the Taylor coefficients of $(z ^{-1}{\bf q}(z))^{1/v}$ are still
integers. In particular, we determine the
Dwork--Kontsevich sequence $(u_N)_{N\ge1}$, where $u_N$ is the
largest integer such that 
$q_N(z)^{1/u_N}$ is a series with integer
coefficients, where $q_N(z)=\exp(G_N(z)/F_N(z))$, 
$F_N(z)=\sum _{m=0} ^{\infty} (Nm)!\,z^m/m!^N$ and
$G_N(z)=\sum _{m=1} ^{\infty} (H_{Nm}-H_m)(Nm)!\,z^m/m!^N$, with $H_n$ denoting
the $n$-th harmonic number, conditional on the conjecture that there
are no prime number $p$ and integer $N$ such that the $p$-adic 
valuation of $H_N-1$ is strictly greater than~$3$.
\end{abstract}

\maketitle

\section{Introduction and statement of results}

The present article is a sequel to our article~\cite{kratrivmirror}, where we proved 
general results concerning the integrality properties of  
mirror maps. We shall prove here stronger integrality assertions for certain special  
cases that appear frequently in the literature.
For any vector $\mathbf{N}=(N, \ldots, N)$ (with $k$ occurrences of
$N$), where $N$ is a positive integer, let us define the power series
\begin{equation}
F_{\mathbf{N}}(z)= \sum_{m=0}^{\infty} 
\frac{(Nm)!^k}{m!^{kN}}  
z^m
\label{eq:FN}
\end{equation}
and
\begin{equation}
G_{\mathbf{N}}(z)=\sum_{m=1}^{\infty} 
kN(H_{Nm}-H_m)
\frac{(Nm)!^k}{m!^{kN}} z^m,
\label{eq:GN}
\end{equation}
with 
$H_n:=\sum _{i=1} ^{n}\frac {1} {i}$ denoting the $n$-th harmonic number.
The functions $F_{\mathbf{N}}(z)$ and
$G_{\mathbf{N}}(z)+\log(z)F_{\mathbf{N}}(z)$ are solutions of the same
hypergeometric differential equation with maximal unipotent monodromy
at $z=0$. A basis of solutions with at most logarithmic
singularities around $z=0$ can then be obtained by Frobenius' method;
see~\cite{yoshida}.  
That differential equation is the Picard--Fuchs equation 
of a one parameter family of mirror manifolds 
$W'$ of a complete intersection $W$ of 
$k$ hypersurfaces $W_1, \ldots, W_k$, all of degree $N$ in
$\mathbb{P}^{d+k}(\mathbb{C})$: $W$ is a family of Calabi--Yau manifolds  
if one chooses
$d$ equal to $k(N-1)-1$. The mirrors $W'$ are
explicitly constructed in~\cite[Sec.~5.2]{batstrat}. In the underlying
context of mirror symmetry, it is natural to  
define the canonical coordinate
$q_{\mathbf{N}}(z):=z\exp\big(G_{\mathbf{N}}(z)/F_{\mathbf{N}}(z)\big)$
and 
the mirror map $z_{\mathbf{N}}(q)$, which is the compositional
inverse of  $q_{\mathbf{N}}(z)$.

In~\cite{kratrivmirror}, we proved the following result, which settled
a conjecture in the folklore of mirror symmetry theory.

\begin{equation} \label{eq:folk} 
\vbox{\hsize14.7cm
{\medskip\leftskip1cm\rightskip1cm
\em\noindent
For any integers $k\ge 1$ and $N\ge 1$, we have 
$q_{\mathbf N}(z)\in z\mathbb{Z}[[z]]$ and 
$z_{\mathbf N}(q)\in q\mathbb{Z}[[q]]$, where  
$N$ is repeated $k$ times in the vector 
$\mathbf N=(N,\ldots, N)$.~$($\footnote{%
In the number-theoretic study undertaken in the present paper,
we are interested in the integrality of the coefficients of roots of mirror
maps $z(q)$. In that context, $z(q)$ and the
corresponding canonical coordinate $q(z)$
play strictly the same role, because $(z^{-1}q(z))^{1/\tau}\in 
1+z\mathbb{Z}[[z]]$ for some integer  
$\tau$ implies that $(q^{-1}z(q))^{1/\tau}\in 1+q\mathbb{Z}[[q]]$, and
conversely; see~\cite[Introduction]{lianyau1}.
We shall, in the sequel, formulate our integrality results exclusively for
canonical coordinates. By abuse of terminology, we 
shall often use the term ``mirror map'' for any  
canonical coordinate.}$)$ 

}
}
\end{equation}

\medskip
Lian and Yau \cite[Sec.~5, Theorem~5.5]{lianyau} had proved earlier
the particular case of  
this theorem where $k=1$ and $N$ is a prime number, and 
Zudilin~\cite[Theorem~3]{zud} had extended their result to any $k\ge 1$
and any $N$ which is a prime power. Zudilin also formulated a more
general conjecture, implying \eqref{eq:folk}, 
which he proved in a 
particular case. The conjecture was subsequently fully proved as one
of the main results in~\cite{kratrivmirror}.

For $k=1$, physicists made the observation that, apparently, even the
stronger assertion
\begin{equation}\label{eq:raflianyau}
\big(z^{-1}q_{(N)}(z)\big)^{1/N}\in \mathbb{Z}[[z]]
\end{equation}
holds.
This was proved by Lian and Yau \cite{lianyau2} for any prime number $N$, 
thus strengthening their result from~\cite{lianyau} mentioned above. 
The observation \eqref{eq:raflianyau} leads naturally to the more
general question of determining the largest integer $V$ 
such that $\left(z ^{-1}q(z) \right) ^{1/V}\in 
\mathbb{Z}[[z]]$ for  mirror maps $q(z)$ such as 
$q_{\mathbf{N}}(z)$.~(\renewcommand\thefootnote{1}\footnote{%
In the number-theoretic study undertaken in the present paper,
we are interested in the integrality of the coefficients of roots of mirror
maps $z(q)$. In that context, $z(q)$ and the
corresponding canonical coordinate $q(z)$
play strictly the same role, because $(z^{-1}q(z))^{1/\tau}\in 
1+z\mathbb{Z}[[z]]$ for some integer  
$\tau$ implies that $(q^{-1}z(q))^{1/\tau}\in 1+q\mathbb{Z}[[q]]$, and
conversely; see~\cite[Introduction]{lianyau1}.
We shall, in the sequel, formulate our integrality results exclusively for
canonical coordinates. By abuse of terminology, we 
shall often use the term ``mirror map'' for any  
canonical coordinate.
\newline
\indent 
$^2$Let $q(z)$ be a given power 
series in $\mathbb{Z}[[z]]$, and let
$V$ be the largest integer with the property that $q(z)^{1/V}\in \mathbb
Z[[z]]$. Then $V$ carries complete information about {\it all\/}
integers with that property: namely,
the set of integers $U$ with $q(z)^{1/U}\in \mathbb
Z[[z]]$ consists of all divisors of $V$. Indeed, it is clear that all
divisors of $V$ belong to this set. Moreover, 
if $U_1$ and $U_2$ belong to this set, then also $\lcm(U_1,U_2)$ does
(cf.\ \cite[Lemma~5]{HeRSAA} for a simple proof 
based on B\'ezout's lemma).}) A rather general result 
in this direction has already
been obtained in \cite[Theorem~2]{kratrivmirror}. 
The purpose of the present paper
is to sharpen this earlier result for wide classes of special
choices of the parameters occurring in \cite{kratrivmirror}.
Indeed, our main results, given in 
Theorems~\ref{thm:3} and \ref{thm:3a} below, provide values of $V$
for infinite families of mirror(-type) maps, which, conditional on
widely believed conjectures on the $p$-adic valuations of $H_N$
respectively $H_N-1$, are optimal for these families.

\medskip

To describe our results,
for positive integers $L$ and $N$, we set 
\begin{equation}
G_{L,\mathbf{N}}(z)=\sum_{m=1}^{\infty} H_{Lm} 
\frac{(Nm)!^k}{m!^{kN}} z^m,
\label{eq:GLN}
\end{equation}
where again $\mathbf N=(N,N,\dots,N)$, with $k$ occurrences of $N$.
We then define  the mirror-type map
$q_{L,\mathbf{N}}(z):=\exp\big(G_{L,\mathbf{N}}(z)/F_{\mathbf{N}}(z)\big)$.
Obviously,  
the mirror map $q_{\mathbf{N}}(z)$ can be expressed as a product
of the series $q_{L,\mathbf{N}}(z)$, namely as
\begin{equation} \label{eq:truemap} 
q_{\mathbf{N}}(z)=
zq_{N,\mathbf N}^{kN}(z)q_{1,\mathbf N}^{-kN}(z).
\end{equation}

The special case of the afore-mentioned Theorem~2 from
\cite{kratrivmirror} where $\mathbf
N=(N,N,\dots,N)$ (with $k$ occurrences of $N$) addressed the above
question of ``maximal integral roots''
for the mirror-type map $q_{L,\mathbf N}(z)$.
It reads as follows:

\begin{equation} \label{eq:NNNN} 
\vbox{\hsize14cm
{\leftskip1cm\rightskip1cm
\em\noindent
Let $\Th_L:=L!/\gcd(L!, L!\,H_L)$ be the denominator of $H_L$
when written as a reduced fraction. For any positive integers $N$ and  
$L$ with $L\le N$, we have 
$q_{L,\mathbf{N}}(z)^{\frac{\Th_L}{N!^k}} \in\mathbb{Z}[[z]].$

}
}
\end{equation}

\medskip
As we remarked in \cite{kratrivmirror}, this result is optimal in the
case that $L=1$; that is, no integer $V$ larger than $N!^k/\Th_1=N!^k$ 
can be found such that $q_{L,\mathbf{N}}(z)^{1/V} \in\mathbb{Z}[[z]]$.
However, if $L\ge2$, improvements may be possible.

Our first main result provides such an improvement. In order to
state the result, we need to introduce usual notation for $p$-adic
valuation: given a prime number $p$ (and from now on, $p$ will always
denote a prime number) and $\alpha\in \mathbf Q_p$, 
$v_p(\alpha)$ denotes the $p$-adic valuation of $\alpha$. 

\begin{theo} \label{thm:3}
Let $N$ be a positive integer, 
$\mathbf N=(N,N,\dots,N)$, with $k$ occurrences of $N$,
and let $\Xi_1=1$, $\Xi_7=1/140$, and, for $N\notin\{1,7\}$,
\begin{equation} \label{eq:Xi} 
\Xi_N:=
{\prod _{p\le N}
^{}}p^{\min\{2+\xi(p,N),v_p(H_N)\}},
\end{equation}
where $\xi(p,N)=1$ if $p$ is a Wolstenholme prime {\em(}i.e.,
a prime $p$ for which $v_p(H_{p-1})\ge3$
{\em(}\renewcommand\thefootnote{3}\footnote{Presently, only two such primes are known, namely 
$16843$ and $2124679$, and it is unknown whether there are infinitely
many Wolstenholme primes or not.}{\em))} or $N$ is divisible by
$p$, and $\xi(p,N)=0$ otherwise. 
Then 
$q_{N,\mathbf N}(z)^{\frac{1}{\Xi_{N}N!^k}} \in\mathbb{Z}[[z]].$
\end{theo}

\begin{Remarks} \label{rem:Xi7}
For better comprehension, we discuss the meaning of the
statement of Theorem~\ref{thm:3} and its implications;
in particular, we address some fine points of the definition of $\Xi_N$.

\smallskip
(a) The case of $N=1$ is trivial since 
$q_{1,(1,\dots,1)}(z)=1/(1-z)$. Furthermore, we have
$$\Xi_7=\frac {1} {140}=2^{v_2(H_7)}5^{v_5(H_7)}7^{v_7(H_7)},$$
which differs by a factor of $3$ from the right-hand side of
\eqref{eq:Xi} with $N=7$ 
(since $v_3(H_7)=v_3(\frac {363} {140})=1$).

\smallskip
(b) Let $V_N$ denotes the largest integer such that 
$\big(z^{-1}q_{N,\mathbf N}(z)\big)^{1/V_N}$ is a series with
integer coefficients.
Since 
$q_{N,\mathbf N}(z)=1+H_N N!^kz+\mathcal O(z^2)$, 
it is clear that $q_{N,\mathbf N}(z)^{1/(p^{v_p(H_N)+1}N!^k)}\notin
\mathbb{Z}[[z]]$, so that the  
exponent of $p$ in the prime factorisation of
$V_{N}$ can be at most $v_p(H_N N!^k)$.
In Theorem~\ref{thm:3}, this theoretically maximal exponent is
cut down to $v_p(\Xi_N N!^k)$. First of all, the number $\Xi_N$ 
contains no prime factor $p>N$, whereas the harmonic number $H_N$ may very well do so
(and, in practice, always does for $N>1$).
Moreover, for primes $p$ with $p\le N$
and $v_p(H_N)\ge3$, the definition of $\Xi_N$ cuts the theoretically
maximal exponent $v_p(H_NN!^k)$
of $p$ down to $2+v_p(N!^k)$ respectively
$3+v_p(N!^k)$, depending on whether $\xi(p,N)=0$ or $\xi(p,N)=1$.
In items~(c)--(e) below, we address the question of how serious this cut is
expected to be.

\smallskip
(c) Clearly, the minimum appearing in the exponent of $p$ in the definition
\eqref{eq:Xi} of $\Xi_N$ is $v_p(H_N)$ as long as $v_p(H_N)\le 2$. 
In other words, the exponent of $p$ in the prime factorisation of
$\Xi_N$ depends largely on the $p$-adic behaviour of $H_N$. 
An extensive discussion of this topic, with many interesting results,
can be found in \cite{boyd}. We have as well computed a table of
harmonic numbers $H_N$ up to 
$N=1000000$.~(\renewcommand\thefootnote{4}\footnote{The summary of the table is available at
{\tt
http://www.mat.univie.ac.at/\~{}kratt/artikel/H.html}.})
Indeed, the data suggest that
pairs $(p,N)$ with $p$ prime, $p\le N$, and $v_p(H_N)\ge 3$ 
are not very frequent. 
More precisely, so far only five examples are
known with $v_p(H_N)=3$: four for $p=11$, with $N=848,9338,10583$,
and $3546471722268916272$, and one for $p=83$
with
\begin{align} \notag
&\hbox{\small
$N=79781079199360090066989143814676572961528399477699516786377994370\backslash$}
\\
&\kern2cm
   \hbox{\small $78839681692157676915245857235055200779421409821643691818$}
\label{eq:boyd} 
\end{align}
(see \cite[p.~289]{boyd}; the value of $N$ in \eqref{eq:boyd}, not
printed in \cite{boyd}, was kindly communicated to us by David Boyd). 
There is no example known with $v_p(H_N)\ge4$.
It is, in fact, conjectured that no $p$ and $N$ exist with
$v_p(H_N)\ge4$. Some evidence for this conjecture (beyond mere
computation) can be found in \cite{boyd}.

\smallskip
(d) Since in all the five examples for which
$v_p(H_N)=3$ we neither have $p\mid N$ (the gigantic number in
\eqref{eq:boyd} is congruent to $42$ modulo $83$) nor that the prime $p$ is a
Wolstenholme prime, the exponent of $p$ in the prime factorisation of
$\Xi_N$ in these cases is $2$ instead of $v_{p}(H_N)=3$. 

\smallskip
(e) On the other hand, should there be a prime $p$ and an integer $N$
with $p\le N$, $v_p(H_N)\ge3$, $p$ a Wolstenholme prime or $p\mid
N$, then the exponent of $p$ in the prime factorisation of
$\Xi_N$ would be $3$. However, no such examples are known. We
conjecture that there are no such pairs $(p,N)$. If this conjecture
should turn out to be true, then, given $N\notin\{1,7\}$,
the definition of $\Xi_N$ in \eqref{eq:Xi} could be simplified to
\begin{equation} \label{eq:Xisimpl} 
\Xi_N:=
{\prod _{p\le N}
^{}}p^{\min\{2,v_p(H_N)\}}.
\end{equation}
\end{Remarks}

Theorem~\ref{thm:3}
improves upon \eqref{eq:NNNN} for $L=N$.
Namely, if one compares the definition of $\Xi_N$ in \eqref{eq:Xi}
with the following alternative way to write the integer 
$\Th_N$ occurring in \eqref{eq:NNNN},
\begin{equation} \label{eq:ThL}
\Th_N=
{\prod _{p\le N}^{}}p^{-\min\{0,v_p(H_N)\}},
\end{equation}
we see that
Theorem~\ref{thm:3} is {\it always} at least as strong as 
\cite[Theorem~2]{kratrivmirror}, 
and it is {\it strictly} stronger if $N\ne7$ and
$v_p(H_N)\ge 1$ for some prime $p$ less than or equal to $N$. 
Indeed, the smallest $N\ne 7$ with that property is $N=20$, in which case 
$v_5(H_{20})=1$.

We remark that strengthenings of
\cite[Theorem~2]{kratrivmirror} in the spirit of Theorem~\ref{thm:3} 
for more general choices of the parameters 
can also be obtained by our techniques but are omitted here. 

We outline the proof of Theorem~\ref{thm:3} in Section~\ref{sec:2},
with the details being filled in 
in Sections~\ref{sec:4}--\ref{sec:aux}.
As we explain in Section~\ref{sec:DworkKont}, 
we conjecture that Theorem~\ref{thm:3} cannot be improved if $k=1$,
that is, that for $k=1$ the largest integer $t_N$ such that
$q_{N,(N)}(z)^{1/t_N}\in\mathbb Z[[z]]$ is exactly $\Xi_N N!$. 
Propositions~\ref{prop:p>N} and \ref{prop:vp=3} 
in Section~\ref{sec:DworkKont} show that 
this conjecture would immediately follow if one could prove
the conjecture from Remarks~\ref{rem:Xi7}(c) above that there are no 
primes $p$ and integers $N$ with $v_p(H_N)\ge 4$. 

\medskip
Even if the series $q_{N,\mathbf N}(z)$ appears in the identity
\eqref{eq:truemap}, which relates the mirror map $q_{\mathbf N}(z)$ 
to the series
$q_{L,\mathbf N}(z)$ (with $L=1$ and $L=N$), Theorem~\ref{thm:3} does not
imply an improvement over \cite[Corollary~1]{kratrivmirror}, which we
recall here for convenience.

\begin{equation} \label{eq:cor1} 
\vbox{\hsize14.7cm
{\leftskip1cm\rightskip1cm
\em\noindent
For all integers $k\ge 1$ and $N\ge 1$, we have 
$
\big(z^{-1}q_{\mathbf N}(z)\big)^{\frac{\Th_N}{N!^k kN}} \in \mathbb{Z}[[z]],
$
where $\Th_N$ is defined in \eqref{eq:NNNN}.

}
}
\end{equation}

\medskip\noindent
The reason is that the
coefficient of $z$ in $(z^{-1}q_{\mathbf N}(z))^{\Th_N/(pN!^kkN)}$ is equal
to $\frac {\Th_N} {p}(H_N-1)$, and, thus,
it will not be integral for primes $p$
with $v_p(H_N)>0$. Still, there is an improvement of
\eqref{eq:cor1} in the spirit of Theorem~\ref{thm:3}. It is our second
main result, and it
involves primes $p$ with $v_p(H_N-1)>0$ instead.

\begin{theo} \label{thm:3a}
Let $N$ be a positive integer with $N\ge2$, and let
\begin{equation} \label{eq:Om} 
\Om_N:=
{\prod _{p\le N}
^{}}p^{\min\{2+\om(p,N),v_p(H_N-1)\}},
\end{equation}
where $\om(p,N)=1$ if $p$ is a Wolstenholme prime 
or $N\equiv\pm1$~{\em mod}~$p$, and $\om(p,N)=0$ otherwise. 
Then 
$\big(z^{-1}q_{\mathbf N}(z)\big)^{\frac{1}{\Om_{N}N!^kkN}} \in\mathbb{Z}[[z]].$
\end{theo}
\begin{Remarks} \label{rem:Om}
Also here, some remarks are in order to get a better understanding of
the above theorem.

\smallskip
(a)
If $v_p(H_N)< 0$, then $v_p(H_N)=v_p(H_N-1)$. Hence, 
differences in the prime factorisations of $\Xi_N$ and $\Om_N$ 
can only arise for primes $p$ with $v_p(H_N)\ge0$.

\smallskip
(b) Since $z^{-1}q_{\mathbf N}(z)=1+(H_N-1)N!^kkNz+\mathcal O(z^2)$, it is
clear that  
$$\big(z^{-1}q_{\mathbf N}(z)\big)^{1/(p^{v_p(H_N-1)+1}N!^kkN)}\notin
\mathbb{Z}[[z]].$$
If $\widetilde V_N$ denotes the largest integer such that 
$\big(z^{-1}q_{\mathbf N}(z)\big)^{1/\widetilde V_NkN}$ is a series with
integer coefficients,
the exponent of $p$ in $\widetilde V_N$ can be at most 
$v_p\big((H_N-1)N!^k\big)$.
In Theorem~\ref{thm:3a}, this theoretically maximal exponent is
cut down to $v_p(\Om_NN!^k)$. Namely, as is the case for $\Xi_N$, 
the number $\Om_N$ in \eqref{eq:Om}
contains no prime factor $p>N$, whereas the difference $H_N-1$ may very well do so
(and, in practice, always does for $N>2$). 
Moreover, for primes $p$ with $p\le N$
and $v_p(H_N-1)\ge3$, the definition of $\Om_N$ cuts the theoretically
maximal exponent $v_p\big((H_N-1)N!^k\big)$
of $p$ down to $2+v_p(N!^k)$ respectively
$3+v_p(N!^k)$, depending on whether $\om(p,N)=0$ or $\om(p,N)=1$.

\smallskip
(c)
Concerning the question whether there are any primes $p$ and integers $N$
with high values of $v_p(H_N-1)$, 
we are not aware of any corresponding literature. 
Our table of harmonic numbers $H_N$ mentioned in Remarks~\ref{rem:Xi7}(c)
does not contain any pair $(p,N)$ with 
$v_p(H_N-1)\ge3$.~(\renewcommand\thefootnote{5}\footnote{The summary of the 
corresponding table, containing pairs $(p,N)$ with $p\le N$ and
$v_p(H_N-1)>0$, is available at
{\tt
http://www.mat.univie.ac.at/\~{}kratt/artikel/H1.html}.})
In ``analogy'' to the conjecture mentioned in
Remarks~\ref{rem:Xi7}(c), we conjecture that no $p$ and $N$ exist with
$v_p(H_N-1)\ge4$. It may even be true that there are no $p$ and
$N$ with $v_p(H_N-1)\ge3$, in which case 
the definition of $\Om_N$ in \eqref{eq:Om} could be simplified to
\begin{equation} \label{eq:Omsimpl} 
\Om_N:=
{\prod _{p\le N}
^{}}p^{\min\{2,v_p(H_N-1)\}}.
\end{equation}
\end{Remarks}

In view of \eqref{eq:ThL}, Theorem~\ref{thm:3a} 
improves upon \eqref{eq:cor1}.
Namely, Theorem~\ref{thm:3a} is {\it always} at least as strong as 
\eqref{eq:cor1}, and it is {\it strictly} stronger if 
$v_p(H_N-1)\ge 1$ for some prime $p$ less than or equal to $N$. 
The smallest $N$ with that property is $N=21$, in which case 
$v_5(H_{21}-1)=1$.

We sketch the proof of Theorem~\ref{thm:3a} in Section~\ref{sec:Om}.
We also explain in that section that
we conjecture that Theorem~\ref{thm:3a} with $k=1$ is optimal,
that is, that for $k=1$ the largest integer $u_N$ such that
$\big(z^{-1}q_{(N)}(z)\big)^{\frac{1}{Nu_N}} \in\mathbb{Z}[[z]]$ 
is exactly $\Om_N N!$. 
Propositions~\ref{prop:p>N2} and \ref{prop:vp=32} 
in Section~\ref{sec:Om} show that 
this conjecture would immediately follow if one could prove
the conjecture from Remarks~\ref{rem:Om}(c) above that there are no 
primes $p$ and integers $N$ with $v_p(H_N-1)\ge 4$. 
As a matter of fact, the sequence $(u_{2N})_{N\ge1}$ appears in the 
On-Line Encyclopedia of Integer Sequences~\cite{oeis},
as sequence {\tt A007757}, contributed around 1995 by R.~E.~Borcherds under 
the denomination ``Dwork--Kontsevich sequence,''
without any reference or explicit formula for it, however.
(\renewcommand\thefootnote{6}\footnote{In private communication, both, Borcherds and Kontsevich
could not remember where exactly this sequence and its denomination came 
from.}) 

\section{Structure of the paper}

We now briefly review the organisation of 
the rest of the paper. 
Following the steps of previous authors, our approach
for proving Theorems~\ref{thm:3} and \ref{thm:3a}
uses $p$-adic analysis. In particular, we make essential
use of Dwork's $p$-adic theory (in the spirit of \cite{lianyau2}). 
Since the details of our proofs are involved, we provide brief
outlines of the proofs of Theorems~\ref{thm:3} and \ref{thm:3a}
in separate sections.
Namely, Section~\ref{sec:2} provides an outline of the proof of
Theorem~\ref{thm:3}, while Section~\ref{sec:Om} contains a sketch of
the proof of Theorem~\ref{thm:3a}.
Both follow closely the chain of arguments
used in the proof of \cite[Theorem~2]{kratrivmirror}.
In the outline, respectively sketch, the proofs are reduced to a certain
number of lemmas. The lemmas which are necessary for the proof of
Theorem~\ref{thm:3} are established in
Sections~\ref{sec:4}--\ref{sec:aux}, while the lemmas which are necessary 
for the proof of Theorem~\ref{thm:3a} are contained in Section~\ref{sec:Om}.
Section~\ref{sec:DworkKont} reports on the
evidence to believe (or not to believe) that the value $t_N$
(defined in the next-to-last paragraph before Theorem~\ref{thm:3a})
is given by $t_N=\Xi_N N!$, $N=1,2,\dots$,
while Section~\ref{sec:DK} addresses 
the question of whether the Dwork--Kontsevich sequence $(u_N)_{N\ge1}$
(defined in the last paragraph of the Introduction) is (or is
not) given by $u_N=\Om_N N!$, $N=1,2,\dots$.

\section{Outline of the proof of Theorem~\ref{thm:3}}\label{sec:2}

In this section, we provide a brief outline of the proof of
Theorem~\ref{thm:3}, reducing it to
Lemmas~\ref{lem:12}--\ref{lem:congH} and Corollaries~\ref{cor:C1} and
\ref{cor:W2}, the proofs of which are postponed to
Sections~\ref{sec:4}--\ref{sec:aux}, except that Lemma~\ref{lem:10}
has already been established in \cite[Lemma~6]{kratrivmirror}. 

The whole proof is heavily based on 
Dwork's $p$-adic theory (as presented by Lian and Yau in \cite{lianyau2}), 
enhanced in the spirit of~\cite{kratrivmirror}. 
In particular, there we used the following result (cf.\ 
\cite[Lemma~10]{kratrivmirror}).

\begin{lem}\label{lem:4} Given two formal power series 
$f(z)\in 1+z\mathbb{Z}[[z]]$ and $g(z)\in z\mathbb{Q}[[z]]$, an
integer $\tau\ge 1$ and a prime number $p$, we have 
$\exp\big(g(z)/(\tau f(z))\big) \in 1+z\mathbb{Z}_p[[z]]$ if and only if 
\begin{equation}\label{eq:uv=pvu}
f(z)g(z^p)-p\,f(z^p)g(z)\in p\tau z\mathbb{Z}_p[[z]].
\end{equation}
\end{lem}
It follows from Lemma~\ref{lem:4} that Theorem~\ref{thm:3} 
can be reduced to the following statement:  for any prime number $p$, 
we have 
$$
F_{\mathbf N}(z)G_{N,\mathbf N}(z^p)-pF_{\mathbf N}(z^p)G_{N,\mathbf
N}(z) \in p
\Xi_N N!^k z \mathbb{Z}_p[[z]].
$$ 

We now follow the  
presentation in~\cite{kratrivmirror} and let $0\le a<p$ and $K\ge 0$. 
The $(a+Kp)$-th Taylor coefficient of
$F_{\mathbf N}(z)G_{N,\mathbf N}(z^p)-pF_{\mathbf N}(z^p)G_{N,\mathbf
N}(z)$ is  
\begin{equation} \label{eq:C} 
C(a+Kp):=\sum_{j=0}^K B_{\mathbf N}(a+jp)B_{\mathbf N}(K-j) 
(H_{N(K-j)}-pH_{Na+Njp}),
\end{equation}
where $B_{\mathbf N}(m)=\frac{(Nm)!^k}{m!^{kN}}$. 
As said above, proving Theorem~\ref{thm:3} turns out to be equivalent to
proving that 
\begin{equation} \label{eq:Ccong} 
C(a+Kp) \in p \Xi_N N!^k\mathbb{Z}_p
\end{equation}
for all primes $p$ and non-negative integers $a$ and $K$ with $0\le a<p$.

There are two cases which can be treated directly:
if $K=a=0$, then $C(0)=0$, and thus \eqref{eq:Ccong} holds trivially,
whereas if $K=0$ and $a=1$, then $C(1)=-pB_{\mathbf
N}(1)H_N=-pN!^kH_N$, and thus \eqref{eq:Ccong} holds by definition of
$\Xi_N$. We therefore assume 
in addition $a+Kp\ge2$ for the remainder of this section.

\medskip
The following lemma will be useful. Its (simple) proof can be found in
\cite[Lemma~4]{kratrivmirror}.

\begin{lem} \label{lem:multinomial/N!}
For all integers $m\ge 1$ and $N\ge 1$, we have 
$$
\frac {(Nm)!} {m!^N} \in N! \,\mathbb{Z}.
$$
\end{lem}

We deduce in particular that $B_{\mathbf N}(m)\in N!^k\mathbb{Z}$ for
any $m\ge 1$. 

\medskip

Since 
$$
H_J = \sum_{j=1} ^{\fl{J/p}} \frac{1}{pj} + 
\underset{p\nmid j}{\sum_{j=1} ^{J}}\;
\frac{1}{j},
$$
we have
\begin{equation} \label{eq:J} 
pH_{J} \equiv H_{\fl{J/p}} \mod p\mathbb{Z}_p.
\end{equation}
Applying this to $J=Na+Njp$, we get
$$
pH_{Na+Njp}\equiv H_{\lfloor Na/p\rfloor+Nj} \mod p\mathbb{Z}_p.
$$
By Lemma~\ref{lem:multinomial/N!}, we infer
\begin{equation}\label{eq:firstreduction1}
C(a+Kp) \equiv \sum_{j=0}^K B_{\mathbf N}(a+jp)B_{\mathbf N}(K-j)
(H_{N(K-j)}-H_{\lfloor 
Na/p\rfloor+Nj}) \mod pN!^k\mathbb{Z}_p.
\end{equation}
Indeed, if $K\ge 1$ or $a\ge1$, this is because 
$a+jp$ and $K-j$ cannot be simultaneously zero and therefore at least one of 
$B_{\mathbf N}(a+jp)$ or $B_{\mathbf N}(K-j)$ is divisible by 
$N!^k$ by 
Lemma~\ref{lem:multinomial/N!}.  

As long as $v_p(H_N)\le0$, Equation~\eqref{eq:firstreduction1} implies
\begin{equation}\label{eq:firstreduction}
C(a+Kp) \equiv \sum_{j=0}^K B_{\mathbf N}(a+jp)B_{\mathbf N}(K-j)
(H_{N(K-j)}-H_{\lfloor 
Na/p\rfloor+Nj}) \mod p\Xi_NN!^k\mathbb{Z}_p.
\end{equation}
We claim that this congruence holds also for $p$ and $N$ with $v_p(H_N)>0$.

For $p=2$, we observe that 
we have always $v_2(H_N)\le 0$ because of Lemma~\ref{lem:H_L}, so that no
improvement over \eqref{eq:firstreduction1} is needed in this case.
Furthermore, Lemma~\ref{lem:3} together with Remarks~\ref{rem:Xi7}(a) 
in the Introduction tells us that, if $p=3$, we need an improvement only if
$N=22$. To be precise, for $N=22$ we need to show that
\begin{equation}\label{eq:firstreductionU3}
C(a+3K) \equiv \sum_{j=0}^K B_{\mathbf N}(a+3j)B_{\mathbf N}(K-j)
(H_{N(K-j)}-H_{\lfloor 
Na/3\rfloor+Nj}) \mod 3^2 N!^k\mathbb{Z}_3.
\end{equation} 
For $p\ge 5$, we should prove
\begin{equation}\label{eq:firstreductionU}
C(a+Kp) \equiv \sum_{j=0}^K B_{\mathbf N}(a+jp)B_{\mathbf N}(K-j)
(H_{N(K-j)}-H_{\lfloor 
Na/p\rfloor+Nj}) \mod p^3 N!^k\mathbb{Z}_p,
\end{equation}
and if, in addition, $v_p(H_N)\ge3$ and $p$ is a Wolstenholme prime or
if $v_p(H_N)\ge3$ and $p\mid N$
(the reader should recall the definition \eqref{eq:Xi} of $\Xi_N$), then
we need to show the even
stronger assertion that
\begin{equation}\label{eq:firstreductionUW}
C(a+Kp) \equiv \sum_{j=0}^K B_{\mathbf N}(a+jp)B_{\mathbf N}(K-j)
(H_{N(K-j)}-H_{\lfloor 
Na/p\rfloor+Nj}) \mod p^4 N!^k\mathbb{Z}_p.
\end{equation}


In order to see \eqref{eq:firstreductionU3}--\eqref{eq:firstreductionUW},
we begin by combining our assumption $v_p(H_N)>0$ with
Corollary~\ref{cor:C1}, Lemma~\ref{lem:p<L}(1), and \eqref{eq:J}, 
to obtain that
$$
B_{\mathbf N}(a+jp)B_{\mathbf N}(K-j)pH_{Na+Njp} 
\equiv B_{\mathbf N}(a+jp)B_{\mathbf N}(K-j)H_{\fl{Na/p}+Nj} \mod
p^3N!^k\mathbb{Z}_p
$$
as long as $a+Kp\ge2$ and $a\ne0$. Moreover, due to
Lemma~\ref{lem:multinomial/N!} and Corollary~\ref{cor:W2}, the above
congruence is even true if $a=0$
and $p\ge5$.
This implies \eqref{eq:firstreductionU}.
The congruence \eqref{eq:firstreductionU3} follows in the same way by
the slightly weaker assertion for $p=3$ in Corollary~\ref{cor:W2}.

For the congruence \eqref{eq:firstreductionUW}, one needs to combine
Corollary~\ref{cor:C1} and Lemma~\ref{lem:p<L}(4), to see
that
\begin{equation} \label{eq:congW}
B_{\mathbf N}(a+jp)B_{\mathbf N}(K-j)pH_{Na+Njp} 
\equiv B_{\mathbf N}(a+jp)B_{\mathbf N}(K-j)H_{\fl{Na/p}+Nj} \mod
p^4N!^k\mathbb{Z}_p
\end{equation}
as long as $a+Kp\ge2$ and $a\ne0$. 
If $a=0$, then the congruence \eqref{eq:congW}
still holds as long as $j<K$ because of 
Lemma~\ref{lem:multinomial/N!}, Corollary~\ref{cor:W2}, and the fact
that the term $B_{\mathbf N}(K-j)$ contributes at least one factor $p$. 
The only remaining case to be discussed
is $a=0$ and $j=K$.
If we apply the simple observation that
$v_p(B_{\mathbf N}(p^eh))=v_p(B_{\mathbf N}(h))$ for any positive
integers $e$ and $h$
to $B_{\mathbf N}(a+jp)=B_{\mathbf N}(jp)$, then, making again appeal to
Corollary~\ref{cor:C1} and Lemma~\ref{lem:p<L}(4),
we see that \eqref{eq:congW} holds as well
as long as $j=K$ is no prime power. Finally, let $a=0$ and $j=K$ be a prime
power, $j=K=p^f$ say. If $f\ge1$, then we may use Lemma~\ref{lem:W3}
with $J=a+jp=a+Kp=p^{f+1}$ and $v_p(B_{\mathbf N}(p^{f+1}))=v_p(B_{\mathbf
N}(1))=v_p(N!^k)$ to conclude that \eqref{eq:congW} also holds in this case.
On the other hand, if $j=K=1$, then Lemma~\ref{lem:congH}, the
assumption that $p$ is a Wolstenholme prime or that $p$ divides $N$,
the fact that $v_p(B_{\mathbf N}(p))=v_p(B_{\mathbf
N}(1))=v_p(N!^k)$,
altogether yield the congruence \eqref{eq:congW} in this case as well.
\medskip  

We now want to transform the sum on the right-hand side 
of~\eqref{eq:firstreduction} to a more manageable expression. 
In particular, we want to get rid of the floor function $\fl{Na/p}$.
In order to achieve this, we will prove the following lemma
in Section~\ref{sec:4}.

\begin{lem} \label{lem:12}
For any prime $p$, non-negative integers $a$, $j$, $K$ with $0\le
a<p$ and $K\ge1$. If $a\ne1$ or $j\ge1$, we have
\begin{equation} \label{eq:congrconj1}
B_{\mathbf N}(a+pj)
\left(H_{Nj+\lfloor Na/p\rfloor} - H_{Nj}\right) \in p
\Xi_N N!^k\mathbb{Z}_p\ ,
\end{equation}
and, in the case that $a=1$ and $j=0$, we have
\begin{equation} \label{eq:congrconj1a}
B_{\mathbf N}(1)B_{\mathbf N}(K)
H_{\lfloor N/p\rfloor}  \in p
\Xi_N N!^k\mathbb{Z}_p\ ,
\end{equation}
where $\mathbf N=(N,N,\dots,N)$, with $k$ occurrences of $N$. 
\end{lem}
 
It follows from Eq.~\eqref{eq:firstreduction} and Lemma~\ref{lem:12} that 
 $$
C(a+Kp) \equiv \sum_{j=0}^K B_{\mathbf N}(a+jp)B_{\mathbf N}(K-j)
\Big(H_{N(K-j)}-H_{Nj}\Big) 
\mod p\Xi_N N!^k\mathbb{Z}_p\ ,
$$
which can be rewritten as 
\begin{equation}\label{eq:beforelemma4.2}
C(a+Kp) \equiv -\sum_{j=0}^K
H_{Nj}\Big(B_{\mathbf N}(a+jp)B_{\mathbf N}(K-j)-B_{\mathbf
N}(j)B_{\mathbf N}(a+(K-j)p)\Big) \mod p\Xi_N N!^k\mathbb{Z}_p\ . 
\end{equation}

We now use a combinatorial lemma due to Dwork
(see~\cite[Lemma~4.2]{dwork}) which provides   
an alternative way to write the sum on the
right-hand side of~\eqref{eq:beforelemma4.2}: namely, we have 
\begin{equation} \label{eq:107}
\sum_{j=0}^K H_{Nj}\Big(B_{\mathbf N}(a+jp)B_{\mathbf
N}(K-j)-B_{\mathbf N}(j)B_{\mathbf N}(a+(K-j)p)\Big) 
=
\sum _{s=0} ^{r}
\sum _{m=0} ^{p^{r+1-s}-1}Y_{m,s},
\end{equation}
where $r$ is such that $K<p^r$, and 
\begin{equation*} 
Y_{m,s}:=\big(H_{Nmp^s}-H_{N\fl{m/p}p^{s+1}}\big)S(a,K,s,p,m),
\end{equation*}
the expression $S(a,K,s,p,m)$ being defined by
$$
S(a,K,s,p,m):=\sum _{j=mp^s} ^{(m+1)p^s-1}\big(B_{\mathbf
N}(a+jp)B_{\mathbf N}(K-j)-B_{\mathbf N}(j)B_{\mathbf N}(a+(K-j)p) 
\big).
$$
In this expression for $S(a,K,s,p,m)$, it is assumed that $B_{\mathbf
N}(n)=0$ for negative integers~$n$. 

\medskip

It would suffice to prove that 
\begin{equation}
\label{eq:yms}
Y_{m,s}\in p\Xi_N N!^k\mathbb Z_p 
\end{equation}
because~\eqref{eq:beforelemma4.2} and \eqref{eq:107}
would then imply that 
$
C(a+Kp) \in  p\Xi_N N!^k\mathbb{Z}_p, 
$
as desired. 

Equation~\eqref{eq:yms} can be proved in the following manner.
The expression for\break $S(a,K,s,p,m)$ is of a form which can be treated by 
Dwork's formal congruence theorem (see~\cite[Theorem~1.1]{dwork}). This 
enables us to get the
following lemma, whose proof is given 
in~\cite{kratrivmirror} and will not be repeated here.
\begin{lem} \label{lem:10}
For all primes $p$ and
non-negative integers $a,m,s,K$ with $0\le a<p$, we have 
\begin{equation}
\label{eq:congruenceS}
S(a,K,s,p,m)\in p^{s+1} B_{\mathbf N}(m) \mathbb{Z}_p.
\end{equation}
\end{lem}
Furthermore, in Section~\ref{sec:6} we shall prove the following lemma.
\begin{lem} \label{lem:11}
For all primes $p$, non-negative integers $m$, 
and positive integers $N$, we have
\begin{equation} 
\label{eq:110}
B_{\mathbf N}(m)\big(H_{Nmp^s}-H_{N\fl{m/p}p^{s+1}}\big)\in p^{-s}
\Xi_N N!^k\mathbb Z_p\ ,
\end{equation}
where again $\mathbf N=(N,N,\dots,N)$, with $k$ occurrences of $N$. 
\end{lem}

It is clear that~\eqref{eq:congruenceS} and~\eqref{eq:110}
imply~\eqref{eq:yms}. 
This completes the outline of the proof of Theorem~\ref{thm:3}.

\section{Proof of Lemma~\ref{lem:12}} \label{sec:4}

The assertion is trivially true if $\lfloor Na/p\rfloor=0$, that is,
if $0\le a<p/N$. We may hence assume that $p/N\le a<p$ from now on.

\subsection{First part: a weak version of Lemma~\ref{lem:12}}
From \cite[Eq.~(9.1)]{kratrivmirror} with 
$\mathbf N=(N,N,\dots,\break N)$, we know that
\begin{equation} \label{eq:congrconj2}
B_{\mathbf N}(a+pj)\left(H_{Nj+\lfloor Na/p\rfloor} - H_{Nj}\right) \in
p\mathbb{Z}_p\ .
\end{equation}
(The reader should note the absence of the term
$\Xi_N N!^k$ in comparison with \eqref{eq:congrconj1} or 
\eqref{eq:congrconj1a}.)
In \cite[Sec.~9.1]{kratrivmirror}, we have also shown that
for any 
$$D\le 1+\max_{1\le\ep\le \lfloor Na/p\rfloor}v_p(Nj+\ep)$$ 
we have
\begin{equation} \label{eq:sumest} 
\sum _{\ell =2} ^{D}\left(\fl{\frac {N(a+pj)} {p^\ell }}-
N\fl{\frac {a+pj} {p^\ell }}\right)\ge D-1,
\end{equation}
a fact that we shall use later on.

\medskip
We now embark on the proof of the lemma in its full form.

\subsection{Second part: the case $j=0$}
In this case, we want to prove that, for $a\ge2$, we have
\begin{equation} \label{eq:congrconj4} 
B_{\mathbf N}(a)
H_{\lfloor Na/p\rfloor} \in p\Xi_N N!^k\mathbb{Z}_p\ ,
\end{equation}
respectively that
\begin{equation} \label{eq:congrconj4a} 
B_{\mathbf N}(1)B_{\mathbf N}(K)
H_{\lfloor N/p\rfloor} \in p\Xi_N N!^k\mathbb{Z}_p\ .
\end{equation}
The reader should keep in mind that we still assume
that $p/N\le a<p$, so that, in particular, $a>0$. 

If $p>N$, then our claim would follow from
$B_{\mathbf N}(a)H_{\lfloor Na/p\rfloor} \in
p\mathbb{Z}_p$, 
which is indeed true because of \eqref{eq:congrconj2} with $j=0$.

Now let $p\le N$. If $a=1$, then use of \eqref{eq:J} in the other
direction yields
$$
B_{\mathbf N}(1)B_{\mathbf N}(K)H_{\lfloor N/p\rfloor}\equiv
pN!^kB_{\mathbf N}(K)H_{N}\mod pN!^kB_{\mathbf N}(K)\mathbb Z_p\ .
$$
By the assumption that $K\ge1$ and 
Lemma~\ref{lem:multinomial/N!}, this congruence implies
\begin{equation} \label{eq:wien2} 
B_{\mathbf N}(1)B_{\mathbf N}(K)H_{\lfloor N/p\rfloor}\equiv
pN!^kB_{\mathbf N}(K)H_{N}\mod pN!^{k+1}\mathbb Z_p\ .
\end{equation}
It is easy to see that
$$
v_p(N!)\ge\begin{cases} 
1,&\text{if }\frac {N} {3}<p\le N,\\
3,&\text{if }p\le \frac {N} {3}.
\end{cases}
$$
We have $v_p(H_N)=-1$ as long as $\frac {N} {3}<p\le N$ (for $p=2$ this
follows from Lemma~\ref{lem:H_L}), except if $p=3$ and $N=7$, in which case we have
$v_3(H_7)=v_p(\frac {363} {140})=1$.
From the definition
of $\Xi_N$ in \eqref{eq:Xi} we see that \eqref{eq:wien2} implies
$$
B_{\mathbf N}(1)B_{\mathbf N}(K)H_{\lfloor N/p\rfloor}\equiv
pN!^kB_{\mathbf N}(K)H_{N}\mod p\Xi_NN!^{k}\mathbb Z_p\ .
$$
Again by the definition of $\Xi_N$, we have
$H_{N}\in \Xi_N\mathbb Z_p$, and, thus,
$$
B_{\mathbf N}(1)B_{\mathbf N}(K)H_{\lfloor N/p\rfloor}\in
 p\Xi_NN!^{k}\mathbb Z_p\ ,
$$
which is exactly \eqref{eq:congrconj4a}.

\medskip
From now on, we assume $a\ge2$.
We bound the $p$-adic valuation of the
expression on the left in \eqref{eq:congrconj4} from below:
\begin{align} \notag
v_p\big(B_{\mathbf N}(a)H_{\fl{Na/p}}\big)
&=
k\sum _{\ell =1} ^{\infty}\left(\fl{\frac {Na} {p^\ell }}-
N\fl{\frac {a} {p^\ell }}\right)+
v_p(H_{\fl{Na/p}})\notag\\
&\ge k\sum _{\ell =1} ^{\infty}\fl{\frac {Na} {p^\ell }}
-\fl{\log_p Na/p}\notag\\
&\ge k\sum _{\ell =1} ^{\infty}\fl{\frac {2N} {p^\ell }}
-\fl{\log_p N}\notag\\
&\ge\fl{\frac {N} {p}}+
k\sum _{\ell =1} ^{\infty}\fl{\frac {N} {p^\ell }}
-\fl{\log_p N}
\label{eq:unglA0}\\
&\ge\fl{N/p}+
k\cdot v_p(N!)
-\fl{\log_p N}.
\label{eq:unglA}
\end{align}
If $p=2$, then we can continue the estimation \eqref{eq:unglA} as
\begin{equation} \label{eq:unglB} 
v_2\big(B_{\mathbf N}(a)H_{\fl{Na/2}}\big)\ge 
1+k\cdot v_2(N!)
-\fl{\log_2 N}=
v_2\big(2\Xi_N N!^k\big),
\end{equation}
where we used Lemma~\ref{lem:H_L} and the definition of $\Xi_N$ 
to obtain the equality.
(In fact, at this point it was not necessary to consider the case
$p=2$ because $a<p$ and because we assumed $a\ge2$. However, we shall
re-use the present estimations later in the third part of the current
proof, in a context where $a=1$ is allowed.)

From now on let $p\ge3$.
We use the fact that
\begin{equation} \label{eq:log}
\fl x\ge\fl{\log_px}+2 
\end{equation}
for all $x\ge2$ and primes $p\ge3$.
If we suppose that  $v_p(H_N)\le 0$, 
or that $p=3$ and $N=7$, then $v_p(\Xi_N)\le0$. 
(The reader should recall that $\Xi_N$ has a separate definition for $N=7$,
Remarks~\ref{rem:Xi7}(a) in the Introduction.)
Hence, in the case that $N\ge2p$,
the estimation \eqref{eq:unglA} can be continued as
$$
v_p\big(B_{\mathbf N}(a)H_{\fl{Na/p}}\big)
\ge 
2+\fl{\log_p N/p}+
k\cdot v_p(N!)
-\fl{\log_p N}
\ge 
v_p(pN!^k)\ge v_p(p\Xi_NN!^k),
$$
implying \eqref{eq:congrconj4} in this case.
If $p\le N<2p$, in which case $v_p(H_N)=-1$, 
the estimation \eqref{eq:unglA} can be continued as
\begin{align*}
v_p\big(B_{\mathbf N}(a)H_{\fl{Na/p}}\big)
&\ge 
1+
k\cdot v_p(N!)
-1=v_p(p\Xi_N N!^k),
\end{align*}
implying \eqref{eq:congrconj4} in this case also.

Now let  $v_p(H_N)> 0$,
but not $p=3$ and $N=7$. (The latter case was already discussed.)
We claim that under this condition we have
\begin{equation} \label{eq:logungl} 
\fl{N/p}-\fl{\log_p N/p}\ge
\begin{cases} 
4,&\text{if }p=5,\\
5,&\text{if }p>5,\text{ or if $p=3$ and $N=22$},\\
6,&\text{if }v_p(H_N)>2.\\
\end{cases}
\end{equation}
Indeed, if $p=3$, then Lemma~\ref{lem:3}
tells us that only the case
$N=22$ needs to be considered, in which case 
the above inequality is trivially true.
If $p=5$ then Lemma~\ref{lem:5} says that $N$ must be at
least $20$ because otherwise $v_5(H_N)\le 0$. The inequality
\eqref{eq:logungl} follows immediately. On the other hand, if $p=11$
then it can be checked 
(by our table mentioned in Remarks~\ref{rem:Xi7}(c) in the
Introduction, for example) that $N$ must be at least $77$ because otherwise
$v_{11}(H_N)\le 0$. If $p$ is different from $3$, $5$, and $11$, 
then, according to
Lemma~\ref{lem:p<L}(3), we must have $N/p\ge5$. In both of the latter
cases, the inequality \eqref{eq:logungl} follows immediately.
Finally, if $v_p(H_N)>2$, then, according to
Lemmas~\ref{lem:5} and \ref{lem:p<L}(4), we must have 
$p>5$ and $N/p\ge6$, thus implying \eqref{eq:logungl}.

The effect of the above improvement over \eqref{eq:log} is that the
estimation \eqref{eq:unglA} may now be continued to result in
$$
B_{\mathbf N}(a)H_{\lfloor Na/p\rfloor} 
\in \begin{cases} p^3N!^k\mathbb{Z}_p&\text{if }p=5,\\
p^5N!^k\mathbb{Z}_p&\text{if }v_p(H_N)>2,\\
p^4N!^k\mathbb{Z}_p&\text{otherwise.}\end{cases}
$$
This implies \eqref{eq:congrconj4} since $v_5(H_N)\le 1$ for $N\ge 5$ 
according to Lemma~\ref{lem:5}.

Thus, \eqref{eq:congrconj1} with $j=0$ is established.

\subsection{Third part: the case $j>0$}
Now let $j>0$.
If $p>N$, then \eqref{eq:congrconj1} reduces to 
\begin{equation*}
B_{\mathbf N}(a+pj)
\left(H_{Nj+\lfloor Na/p\rfloor} - H_{Nj}\right) \in
p\mathbb{Z}_p\ ,
\end{equation*}
which is 
again true because of \eqref{eq:congrconj2}.

Now let $p\le N$. The reader should keep in mind that we still assume
that $p/N\le a<p$, so that, in particular, $a>0$. 
In a similar way as we did for the expression on the left in
\eqref{eq:congrconj2}, we bound the $p$-adic valuation of the
expression on the left in \eqref{eq:congrconj1} from below. For the sake of
convenience, we write $T_1$ for $\max_{1\le\ep\le \lfloor
Na/p\rfloor}v_p(Nj+\ep)$ and $T_2$ for $\fl{\log_p (a+pj)}$. 
Since it is somewhat hidden where our assumption $j>0$ enters the
subsequent considerations, we point out to the reader that
$j>0$ implies that $T_2\ge1$; without this property the split of
the sum over $\ell$ into subsums in the 
chain of inequalities below would be impossible.
So, using the above notation, we have (the detailed explanations for the
various steps are given immediately after the 
following chain of estimations)
{\allowdisplaybreaks
\begin{align} \notag
v_p\Big(&B_{\mathbf N}(a+pj)\left(H_{Nj+\lfloor Na/p\rfloor} -
H_{Nj}\right)\Big)\\ 
\notag
&=
k\sum _{\ell =1} ^{\infty}\left(\fl{\frac {N(a+pj)}
{p^\ell }}- 
N\fl{\frac {a+pj} {p^\ell }}\right)+
v_p\big(H_{Nj+\lfloor Na/p\rfloor} - H_{Nj}\big)\\
&=
\fl{\frac {N(a+pj)} {p }}-
N\fl{\frac {a+pj} {p }}
+
\sum _{\ell =2} ^{\min\{1+T_1,T_2\}}
\left(\fl{\frac {N(a+pj)} {p^\ell }}-
N\fl{\frac {a+pj} {p^\ell }}\right)
\notag\\
\notag
&\kern1cm
+
\sum _{\ell =\min\{1+T_1,T_2\}+1} ^{\infty}
\left(\fl{\frac {N(a+pj)} {p^\ell }}-
N\fl{\frac {a+pj} {p^\ell }}\right)\\
\notag
&\kern1cm
+
(k-1)\sum _{\ell =1} ^{\infty}\left(\fl{\frac {N(a+pj)}
{p^\ell }}- 
N\fl{\frac {a+pj} {p^\ell }}\right)
+v_p\big(H_{Nj+\lfloor Na/p\rfloor} - H_{Nj}\big)
\notag
\\
\notag
&\ge
\fl{\frac {Na} {p }}-N\fl{\frac {a} {p }}+\min\{1+T_1,T_2\}-1
+
k\sum _{\ell =T_2+1} ^{\infty}
\left(\fl{\frac {N(a+pj)} {p^\ell }}-
N\fl{\frac {a+pj} {p^\ell }}\right)\\
&\kern1cm
+v_p\big(H_{Nj+\lfloor Na/p\rfloor} - H_{Nj}\big)
\label{eq:ungl1}
\\
&\ge
\fl{\frac {Na} {p }}
+T_1+v_p\big(H_{Nj+\lfloor Na/p\rfloor} - H_{Nj}\big)
+\min\{0,T_2-T_1-1\}
\notag\\
&\kern1cm
+
k\sum _{\ell =\fl{\log_p(a+pj)}+1}
^{\infty}\left(\fl{\frac {N(a+pj)} {p^\ell }}- 
N\fl{\frac {a+pj} {p^\ell }}\right)
\label{eq:ungl2}\\
&\ge 
\fl{N/p}
+\min\{0,T_2-T_1-1\}
+
k\sum _{\ell =1} ^{\infty}\fl{\frac {N} {p^\ell
}\cdot\frac {a+pj} {p^{\fl{\log_p(a+pj)}}}}
\label{eq:ungl3}\\
&\ge 
\fl{N/p}+\fl{\log_p (a+pj)}
-\fl{\log_p\big(Nj+\lfloor Na/p\rfloor\big)}-1
+
k\sum _{\ell =1} ^{\infty}\fl{\frac {N} {p^\ell
}\cdot\frac {a+pj} {p^{\fl{\log_p(a+pj)}}}}
\label{eq:ungl4}
\\
&\ge 
\fl{N/p}+\fl{\log_p j}
-\fl{\log_p\big(Nj+\lfloor Na/p\rfloor\big)}
+
k\sum _{\ell =1} ^{\infty}\fl{\frac {N} {p^\ell }}
\label{eq:ungl5}
\\
&\ge 
\fl{N/p}+\fl{\log_p j}
-\fl{\log_p N}-\fl{\log_p\left(j+\frac {1} {N}\lfloor Na/p\rfloor\right)}-1
+
k\cdot v_p(N!)
\label{eq:ungl6}\\
&\ge 
\fl{N/p}
-\fl{\log_p N}-1+v_p(N!^k).
\label{eq:ungl7}
\end{align}
}%
Here, we used \eqref{eq:sumest} in order to get \eqref{eq:ungl1}.
To get \eqref{eq:ungl3}, we used the inequalities
\begin{equation} \label{eq:ungl6a} 
\fl{\frac {Na} {p}}\ge\fl{\frac {N} {p}}
\end{equation}
and 
\begin{equation} \label{eq:ungl100} 
T_1+v_p\big(H_{Nj+\lfloor Na/p\rfloor} - H_{Nj}\big)\ge0.
\end{equation}
To get \eqref{eq:ungl4}, we used that
$$T_2-T_1-1\ge \fl{\log_p (a+pj)}
-\fl{\log_p\big(Nj+\lfloor Na/p\rfloor\big)}-1$$
and
$$
\fl{\log_p (a+pj)}
-\fl{\log_p\big(Nj+\lfloor Na/p\rfloor\big)}-1=
\fl{\log_p j}-\fl{\log_p\big(Nj+\lfloor Na/p\rfloor\big)}
\le 0,
$$
so that
\begin{equation} \label{eq:ungl6b}
\min\{0,T_2-T_1-1\}\ge 
\fl{\log_p (a+pj)}
-\fl{\log_p\big(Nj+\lfloor Na/p\rfloor\big)}-1.
\end{equation}
Next, to get \eqref{eq:ungl5}, we used
\begin{equation} \label{eq:ungl6d}
\fl{\frac {N} {p^\ell }\cdot\frac {a+pj}
{p^{\fl{\log_p(a+pj)}}}}\ge
\fl{\frac {N} {p^\ell }}.
\end{equation}
To get \eqref{eq:ungl6}, we used
\begin{equation} \label{eq:ungl6c}
\fl{\log_p\big(Nj+\lfloor Na/p\rfloor\big)}
\le\fl{\log_p N}+\fl{\log_p\left(j+\frac {1} {N}\lfloor Na/p\rfloor\right)}+1 .
\end{equation}
Finally, we used $\frac {1} {N}\lfloor
Na/p\rfloor<1$ in order to get
\eqref{eq:ungl7}. 

If we now repeat the arguments after \eqref{eq:unglA},
then we see that the estimation \eqref{eq:ungl7} implies
\begin{equation} \label{eq:ungl8} 
v_p\Big(B_{\mathbf N}(a+pj)\left(H_{Nj+\lfloor Na/p\rfloor} - H_{Nj}\right)\Big)
\ge
v_p\big(\Xi_N N!^k\big).
\end{equation}
This almost proves \eqref{eq:congrconj1}, our lower
bound on the $p$-adic valuation of the number in \eqref{eq:congrconj1}
is just by $1$ too low. 

In order to establish that \eqref{eq:congrconj1} is indeed correct, we
assume by contradiction that all the inequalities in the estimations
leading to \eqref{eq:ungl7} and finally to \eqref{eq:ungl8} 
are in fact equalities. 
In particular, the estimations in \eqref{eq:ungl6a} hold with equality only
if $a=1$, which we shall assume henceforth.

If we examine the arguments after \eqref{eq:unglA} that led us from
\eqref{eq:ungl7} to \eqref{eq:ungl8}, then we see that they prove in
fact 
\begin{equation} \label{eq:ungl11} 
v_p\Big(B_{\mathbf N}(a+pj)\left(H_{Nj+\lfloor Na/p\rfloor} -
H_{Nj}\right)\Big) 
\ge
1+v_p\big(\Xi_N N!^k\big)
\end{equation}
except if $v_p(H_N)\le 0$ and:

\bigskip
{\sc Case 1:} $p=2$ and $\fl{N/2}=1$;

{\sc Case 2:} $p\ge3$ and $p\le N<2p$;

{\sc Case 3:} $p=3$ and $\fl{N/3}=2$;

\bigskip

\noindent
Indeed, if $N\ge 3p$, this is obvious while,
 if $2p\le N<3p$, we have $v_p(H_N)=-1$, except if $p=3$ and $N=7$, a case included
in Case~3. Therefore,
if we exclude the case where $p=3$ and $N=7$, then, if $2p\le N<3p$,
we can continue \eqref{eq:ungl7} as
\begin{multline*}
v_p\Big(B_{\mathbf N}(a+pj)\left(H_{Nj+\lfloor Na/p\rfloor} -
H_{Nj}\right)\Big)\\
\ge 
2+\fl{\log_p N/p}-\fl{\log_p N}-1+ v_p(N!^k)
\ge 
v_p(N!^k)\ge v_p(p\Xi_NN!^k),
\end{multline*}
where we used \eqref{eq:log} in the first step. This is exactly \eqref{eq:ungl11}.

We now show that \eqref{eq:ungl11} holds as well in Cases~1--3, thus
completing the proof of \eqref{eq:congrconj1}.

\medskip
{\sc Case 1}. Let first $p=2$ and $N=2$. We then have
\begin{align*}
\min\{0,T_2-T_1-1\}&=\min\{0,\fl{\log_2(2j+1)}-v_2(2j+1)-1\}\\
&=
\min\{0,\fl{\log_2(2j+1)}-1\}=0>-1,
\end{align*}
in contradiction to having equality in \eqref{eq:ungl6b}.

On the other hand, if $p=2$ and $N=3$, we have
$$H_{Nj+\lfloor Na/p\rfloor} - H_{Nj}
=H_{3j+1} - H_{3j}=\frac {1} {3j+1}.$$
If there holds equality in \eqref{eq:ungl6b}, then $Nj+\fl{Na/p}=3j+1$ 
must be a power of $2$, say $3j+1=2^e$ or, equivalently, 
$j=(2^e-1)/3$. It follows that
$$
\fl{\frac {N} {p }\cdot\frac {a+pj}
{p^{\fl{\log_p(a+pj)}}}}
=\fl{\frac {3} {2 }\cdot\frac {1+2j}
{2^{\fl{\log_2(1+2j)}}}}
=\fl{\frac {3} {2 }\cdot\frac {2^{e+1}+1}
{3\cdot 2^{e-1}}}
=2>1=\fl{\frac {3} {2}}=\fl{\frac {N} {p}},
$$
in contradiction to having equality in \eqref{eq:ungl6d} with $\ell=1$.

\medskip
{\sc Case 2}. Our assumptions $p\ge3$ and $p\le N<2p$ imply 
$$
H_{Nj+\lfloor Na/p\rfloor} - H_{Nj}
=H_{Nj+1} - H_{Nj}=\frac {1} {Nj+1}.
$$
Arguing as in the previous case, in order to have equality in 
\eqref{eq:ungl6b}, we must have $Nj+1=f\cdot p^e$ for some positive
integers $e$ and $f$ with $0<f<p$. Thus, $j=(f\cdot p^e-1)/N$ and
$p<N$. (If $p=N$ then $j$ would be non-integral.) It follows that
\begin{equation} \label{eq:ungl12}
\fl{\frac {N} {p }\cdot\frac {a+pj}
{p^{\fl{\log_p(a+pj)}}}}=
\fl{\frac {N} {p }\cdot\frac {f\cdot p^{e+1}+N-p}
{N\cdot p^{\fl{\log_p((f\cdot p^{e+1}+N-p)/N)}}}}. 
\end{equation}
If $f=1$, then we obtain from \eqref{eq:ungl12} that
$$
\fl{\frac {N} {p }\cdot\frac {a+pj}
{p^{\fl{\log_p(a+pj)}}}}=
\fl{\frac {N} {p }\cdot\frac {p^{e+1}+N-p}
{N\cdot p^{e-1}}}
\ge\fl{\frac {p^{e+1}+N-p} {p^{e}}}
>1=\fl{\frac {N} {p}}, 
$$
in contradiction with having equality in \eqref{eq:ungl6d} with $\ell=1$.
 
On the other hand, if $f\ge2$, then we obtain from \eqref{eq:ungl12} that
$$
\fl{\frac {N} {p }\cdot\frac {a+pj}
{p^{\fl{\log_p(a+pj)}}}}\ge
\fl{\frac {f\cdot p^{e+1}+N-p}
{p^{e+1}}}\ge f>1=\fl{\frac {N} {p}}, 
$$
again in contradiction with having equality in \eqref{eq:ungl6d} with $\ell=1$.

\medskip
{\sc Case 3}. Our assumptions $p=3$ and $\fl{N/3}=2$ imply 
$$
H_{Nj+\lfloor Na/p\rfloor} - H_{Nj}
=H_{Nj+2} - H_{Nj}=\frac {1} {Nj+1}+\frac {1} {Nj+2}.
$$
Similar to the previous cases, in order to have equality in 
\eqref{eq:ungl6b}, we must have $Nj+\ep=f\cdot 3^e$ for some positive
integers $\ep,e,f$ with $0<\ep,f<3$. The arguments from Case~2 can now
be repeated almost verbatim. We leave the details to the reader.

\medskip
This completes the proof of the lemma.

\section{Proof of Lemma~\ref{lem:11}} \label{sec:6}

The claim is trivially true if $p$ divides $m$.
We may therefore assume that $p$ does not divide $m$ for the rest of the proof.
Let us write $m=a+pj$, with $0< a<p$. Then
comparison with \eqref{eq:congrconj1} shows that we are in a
very similar situation here. Indeed, we may derive \eqref{eq:110}
from Lemma~\ref{lem:12}. In order to see this, we observe that
\begin{align*}
H_{Nmp^s}-H_{N\fl{m/p}p^{s+1}}&=
\sum _{\ep=1} ^{Nap^s}\frac {1} {Njp^{s+1}+\ep}\\
&=
\sum _{\ep=1} ^{\fl{Na/p}}\frac {1} {Njp^{s+1}+\ep p^{s+1}}
+
\underset{p^{s+1}\nmid \ep}{\sum _{\ep=1} ^{Nap^s}}\frac {1}
{Njp^{s+1}+\ep}\\
&=\frac {1} {p^{s+1}}(H_{Nj+\fl{Na/p}}-H_{Nj})+
\underset{p^{s+1}\nmid \ep}{\sum _{\ep=1} ^{Nap^s}}\frac {1}
{Njp^{s+1}+\ep}.
\end{align*}
Because of $v_p(x+y)\ge\min\{v_p(x),v_p(y)\}$, this implies
$$
v_p(H_{Nmp^s}-H_{N\fl{m/p}p^{s+1}})\ge
\min\{ -1-s+v_p(H_{Nj+\fl{Na/p}}-H_{Nj}),-s\}.
$$
It follows that
\begin{multline}
v_p\Big(B_{\mathbf N}(m)\big(H_{Nmp^s}-H_{N\fl{m/p}p^{s+1}}\big)\Big)\\
\ge
-1-s+\min\left\{v_p\big(B_{\mathbf N}(a+pj)(H_{Nj+\fl{Na/p}}-H_{Nj})\big),
1+v_p\big(B_{\mathbf N}(a+pj)\big)\right\}.
\label{eq:wien1}
\end{multline}
If $v_p(H_N)\le 0$ or $p>N$, 
then use of Lemmas~\ref{lem:multinomial/N!} and \ref{lem:12}
(cf.\ \eqref{eq:congrconj1})
implies \eqref{eq:110} immediately since $v_p(\Xi_N)\le 0$ in this case. 
On the other hand, if $v_p(H_N)>0$ and $p\le N$,
a combination of Corollary~\ref{cor:C1} with Lemma~\ref{lem:p<L}(2) 
yields that $v_p\big(B_{\mathbf N}(a+pj)\big)\ge3+v_p(N!^k)$ as long
as $p\ne3$ and $m\ge2$. If this is used in \eqref{eq:wien1} together
with Lemma~\ref{lem:12}, we obtain
again \eqref{eq:110} under the above restriction on $p$ and $m$. 
If $p=3$, $N=7$, and $m\ge2$, another use of Corollary~\ref{cor:C1}
yields $v_p\big(B_{\mathbf N}(a+pj)\big)\ge2+v_p(N!^k)$, which, together
with \eqref{eq:wien1} and Lemma~\ref{lem:12}, 
implies \eqref{eq:110} also in this case
(recall Remarks~\ref{rem:Xi7}(a)). If $p=3$, $N\ne7$, and $m\ge2$, then either
$N=22$, in which case $v_p\big(B_{\mathbf N}(a+pj)\big)\ge7+v_p(N!^k)$
according to Corollary~\ref{cor:C1}, or $v_p(H_N)\le 0$ according to
Lemma~\ref{lem:3}. In both cases, the claimed congruence \eqref{eq:110}
follows immediately if one combines this with \eqref{eq:wien1}.

In the remaining case $m=1$, we compute directly:
\begin{align*}
B_{\mathbf N}(1)H_{Np^s}
&=N!^k\Bigg(\frac {1} {p^s}H_N+
\frac {1} {p^{s-1}}
\underset{p\nmid\ep}{\sum _{\ep=1} ^{Np}}\frac {1} {\ep}+
\frac {1} {p^{s-2}}
\underset{p\nmid\ep}{\sum _{\ep=1} ^{Np^2}}\frac {1} {\ep}+
\underset{p^{s-2}\nmid \ep}{\sum _{\ep=1} ^{Np^s}}\frac {1} {\ep}\Bigg).
\end{align*}
The last sum over $\ep$ is clearly an element of
$p^{-s+3}\mathbb Z_p$. Moreover, 
if $p\ge5$, Lemma~\ref{lem:W1} implies that the
next-to-last expression between parentheses 
is an element of $p^{-s+4}\mathbb Z_p\subset p^{-s+3}\mathbb Z_p$,
and that the second expression between parentheses is an element of
$p^{-s+3}\mathbb Z_p$.
If $p=3$, it can be checked directly that 
the expression between parentheses is an element of $3^{-s+2}\mathbb Z_3$.
Putting everything together, we conclude that
$$
B_{\mathbf N}(1)H_{Np^s}\in p^{-s} N!^k\Xi_N\mathbb Z_p,
$$
which finishes the proof of \eqref{eq:110}.

\section{Auxiliary lemmas}
\label{sec:aux}

In this section, we prove the auxiliary results necessary 
for the proof of Theorem~\ref{thm:3}, of which the outline was given in 
Section~\ref{sec:2}, with the proofs of Lemmas~\ref{lem:12} and
\ref{lem:11} given in Sections~\ref{sec:4} and \ref{sec:6},
respectively. These are, on the one hand,
improvements of Lemma~\ref{lem:multinomial/N!}
(see Lemmas~\ref{lem:B1}, \ref{lem:B2} and Corollary~\ref{cor:C1}), 
and, on the other hand, assertions addressing specific $p$-adic properties
of harmonic numbers 
(see Lemmas~\ref{lem:H_L}--\ref{lem:congH} and Corollary~\ref{cor:W2}).
Some of the results of this section are also referred to in the next section.

We begin by two lemmas improving on Lemma~\ref{lem:multinomial/N!}.

\begin{lem} \label{lem:B1}
For all positive integers $N$, $a$, and primes $p$ with $2\le a<p$, we
have
$$v_p(B_{\mathbf N}(a))\ge \fl{\frac {N} {p}}+v_p(N!^k),$$
where $\mathbf N=(N,N,\dots,N)$ {\em(}with $k$
occurrences of $N${\em)}.
\end{lem}

\begin{proof}
This was implicitly proved by the estimations leading to
\eqref{eq:unglA0}.
\end{proof}

\begin{lem} \label{lem:B2}
For all positive integers $N$, $a$, $j$, and primes $p$ with $1\le a<p$, we
have
$$v_p(B_{\mathbf N}(a+jp))\ge \fl{\frac {N} {p}}+\min\{1+T_1,T_2\}-1+
v_p(N!^k),$$
where $\mathbf N=(N,N,\dots,N)$ {\em(}with $k$
occurrences of $N${\em)}, and where
$$T_1=\max_{1\le\ep\le \lfloor
Na/p\rfloor}v_p(Nj+\ep)\quad \text{and}\quad  T_2=\fl{\log_p (a+pj)}.$$
\end{lem}

\begin{proof}
This is seen by going through the estimations leading to
\eqref{eq:ungl7}, without employing \eqref{eq:ungl100}
and \eqref{eq:ungl6b}.
\end{proof}

As a corollary to Lemmas~\ref{lem:B1} and \ref{lem:B2}, we obtain the
following succinct $p$-adic estimation for $B_{\mathbf N}(m)$,
which is needed in the proofs of \eqref{eq:firstreductionU},
\eqref{eq:firstreductionUW}, and \eqref{eq:110}.

\begin{coro} \label{cor:C1}
Let $N$ and $m$ be positive integers and $p$ be a prime such that
$m$ is at least $2$ and not divisible by $p$. Then we
have
$$v_p(B_{\mathbf N}(m))\ge \fl{\frac {N} {p}}+v_p(N!^k),$$
where $\mathbf N=(N,N,\dots,N)$ {\em(}with $k$
occurrences of $N${\em)}.
\end{coro}

The next three lemmas of this section 
provide elementary information on the $p$-adic
valuation of harmonic numbers for $p=2,3,5$ which is needed
in the proof of Lemma~\ref{lem:p<L} and is also referred to
frequently at other places. 
(For example, Lemma~\ref{lem:H_L} was used in the proof of~\eqref{eq:unglB}.)
The proofs are not
difficult (cf.\ \cite{boyd}) and are therefore omitted. 

\begin{lem} \label{lem:H_L}
For all positive integers $N$, we have $v_2(H_N)=-\fl{\log_2N}$.
\end{lem}

\begin{lem} \label{lem:3}
We have $v_3(H_2)=v_3(H_7)=v_3(H_{22})=1$.
For positive integers $N\notin\{2,7,22\}$, we have $v_3(H_N)\le0$.
\end{lem}

\begin{lem} \label{lem:5}
We have $v_5(H_4)=2$ and $v_5(H_{20})=v_5(H_{24})=1$.
For positive integers $N\notin\{4,20,24\}$, we have $v_5(H_N)\le0$.
\end{lem}

Next, we record some properties of integers $N$ and primes $p$
for which $v_p(H_N)>0$. These are needed throughout
Section~\ref{sec:2}. 

\begin{lem} \label{lem:p<L}
Let $p$ be a prime, and let $N$ be an integer with $p\le N$.
Then the following assertions hold true:

\begin{enumerate} 
\item If $v_p(H_N)>0$ then $N\ge 2p$.
\item If $v_p(H_N)>0$ and $p\ne 3$ then $N\ge 3p$.
\item If $v_p(H_N)>0$ and $p\notin\{3,5,11\}$ then $N\ge 5p$.
\item If $v_p(H_N)>2$ then $N\ge 6p$.
\end{enumerate}
\end{lem}

\begin{proof}
We have
$$H_N=\frac {1} {p}H_{\fl{N/p}}+
\underset{p\nmid\ep}{\sum _{\ep=1} ^{N}}\frac {1} {\ep}.$$
Since the sum over $\ep$ is in $\mathbb Z_p$, in order to have
$v_p(H_N)>0$ we must have $v_p(H_{\fl{N/p}})>0$. 
Clearly, $v_p(H_1)=0$ so that (1) follows. If $\fl{N/p}<3$ and $p\ne3$, 
then $v_p(H_{\fl{N/p}})$ cannot be positive since
$H_1=1$ and $H_2=\frac {3} {2}$. This implies (2). 
If $\fl{N/p}<5$ and $p\notin \{3,5,11\}$, 
then, again, $v_p(H_{\fl{N/p}})$ cannot be positive since, as we
already noted, $H_1=1$ and $H_2=\frac {3} {2}$, and since
$H_3=\frac {11} {6}$ and $H_4=\frac {25} {12}.$
This yields (3).

To see (4), we observe that, owing to
Lemmas~\ref{lem:H_L}--\ref{lem:5}, we may assume that
$p\notin\{2,3,5\}$. Furthermore, if $p=11$ then, according to our
table referred to in Remarks~\ref{rem:Xi7}(c) in the Introduction
(see also \cite{boyd}), 
we have $N\ge848$. Similarly, if $p=137$ then, according to
our table, we have $N>500000$. The claim can now be established in
the style of the proofs of (1)--(3) upon observing
that $H_5=\frac {137} {60}$.
\end{proof}

We turn to a slight generalisation of Wolstenholme's theorem
on harmonic numbers. 
(We refer the reader to \cite[Chapter~VII]{hw} for information on 
this theorem, which corresponds to the case $r=1$ in the lemma
below.) 

\begin{lem} \label{lem:W1}
For all primes $p\ge5$ and positive integers $r$, we have
$$v_p(H_{rp-1}-H_{rp-p})\ge2.$$
\end{lem}
\begin{proof}
By simple rearrangement, we have
\begin{align*}
H_{rp-1}-H_{rp-p}&=
\sum _{\ep=1} ^{p-1}\frac {1} {rp-p+\ep}=
\sum _{\ep=1} ^{(p-1)/2}\left(\frac {1} {rp-p+\ep}+\frac {1}
{rp-\ep}\right)\\
&=
p(2r-1)\sum _{\ep=1} ^{(p-1)/2}\frac {1} {(rp-p+\ep)(rp-\ep)}.
\end{align*}
It therefore suffices to consider the last sum over $\ep$ in 
$\mathbb Z/p\mathbb Z$ (with $1/\al$ being interpreted as the
multiplicative inverse of $\al$ in $\mathbb Z/p\mathbb Z$)
and show that it is 
congruent to $0$ modulo $p$. If we reduce
this sum 
modulo $p$, then we are left with
$$-\sum _{\ep=1} ^{(p-1)/2}\frac {1} {\ep^2},$$
which is, up to the sign, the sum of all quadratic residues in
$\mathbb Z/p\mathbb Z$, 
that is, equivalently, 
$$-\sum _{\ep=1} ^{(p-1)/2}\ep^2=\frac {p(p-1)(p+1)} {24}.$$
Clearly, this is divisible by $p$ for all primes $p\ge5$.
\end{proof}

As a corollary, we obtain strengthenings of \eqref{eq:J} that we
need in the proof of \eqref{eq:firstreductionU} and
\eqref{eq:firstreductionU3}.

\begin{coro} \label{cor:W2}
For all primes $p\ge5$ and positive integers $J$ divisible by $p$, we have
$$
pH_{J} \equiv H_{{J/p}} \mod p^3\mathbb{Z}_p.
$$
Moreover, for all positive integers $J$ divisible by $3$, we have
$$
3H_{J} \equiv H_{{J/3}} \mod 3^2\mathbb{Z}_3.
$$
\end{coro}
\begin{proof}
By simple rearrangement, we have
$$pH_{J} - H_{{J/p}}= 
p\sum _{r=1} ^{J/p}(H_{rp-1}-H_{rp-p}).$$
Due to Lemma~\ref{lem:W1}, the $p$-adic valuation of this
expression is at least $3$
if $p\ge5$. If $p=3$, it is easily seen
directly that the $3$-adic valuation of this expression is at least $2$.
\end{proof}

Further strengthenings of \eqref{eq:J}, needed 
in the proof of \eqref{eq:firstreductionUW}, are given in the final
two lemmas of this section.

\begin{lem} \label{lem:W3}
For all primes $p\ge5$ and positive integers $J$ divisible by $p^2$, we have
$$
pH_{J} \equiv H_{{J/p}} \mod p^5\mathbb{Z}_p.
$$
\end{lem}
\begin{proof}
Again, by simple rearrangement, we have
\begin{align*}
pH_{J}-H_{J/p}&=p
\underset{p\nmid\ep}{\sum _{\ep=1} ^{J-1}}\frac {1} {\ep}
=p\sum _{r=1} ^{J/p^2}
\underset{p\nmid\ep}{\sum _{\ep=1} ^{p^2-1}}\frac {1} {rp^2-p^2+\ep}\\
&=p^3\sum _{r=1} ^{J/p^2}(2r-1)
\underset{p\nmid\ep}{\sum _{\ep=1} ^{(p^2-1)/2}}
\frac {1} {(rp^2-p^2+\ep)(rp^2-\ep)}.
\end{align*}
It therefore suffices to consider the last sum over $\ep$ in 
$\mathbb Z/p^2\mathbb Z$
and show that it is 
congruent to $0$ modulo $p^2$. If we reduce
this sum 
modulo~$p^2$, then we are left with
$$-\underset{p\nmid\ep}{\sum _{\ep=1} ^{(p^2-1)/2}}\frac {1} {\ep^2},$$
which is, up to the sign, the sum of all quadratic residues in
$\mathbb Z/p^2\mathbb Z$, 
that is, equivalently, 
$$-\underset{p\nmid\ep}{\sum _{\ep=1} ^{(p^2-1)/2}}\ep^2
=-\sum _{\ep=1} ^{(p^2-1)/2}\ep^2
+\sum _{\ep=1} ^{(p-1)/2}(p\ep)^2
=-\frac {p^2(p^2-1)(p^2+1)} {24}+p^2
\frac {p(p-1)(p+1)} {24}.$$
Clearly, this is divisible by $p^2$ for all primes $p\ge5$.
\end{proof}

\begin{lem} \label{lem:congH}
For all primes $p\ge5$ and positive integers $N$, we have
\begin{equation} \label{eq:congH1}
pH_{pN}\equiv H_N\mod p^4\mathbb Z_p
\end{equation}
if and only if $p$ is a Wolstenholme prime or $p$ divides $N$.
\end{lem}
\begin{proof}
Using a rearrangement in the spirit of Lemma~\ref{lem:W1}, we obtain
$$pH_{pN}-H_N=
p\sum _{r=1} ^{N}\sum _{\ep=1} ^{p-1}\frac {1} {rp-p+\ep}.
$$
We consider the sum over $r$ in $\mathbb Z/p^3\mathbb Z$. This leads
to
\begin{align} \notag
\sum _{r=1} ^{N}\sum _{\ep=1} ^{p-1}(rp-p+\ep)^{-1}
&\equiv
\sum _{r=1} ^{N}\sum _{\ep=1} ^{p-1}\ep^{-1}\big(1+p(r-1)\ep^{-1}\big)^{-1}\\
\notag
&\equiv
\sum _{r=1} ^{N}\sum _{\ep=1} ^{p-1}
\left(\ep^{-1}-p(r-1)\ep^{-2}+p^2(r-1)^2\ep^{-3}\right)\\
&\equiv NH_{p-1}-p\binom N2H_{p-1}^{(2)}+p^2\frac {N(N-1)(2N-1)}
{6}H_{p-1}^{(3)}\kern-5pt\mod \mathbb Z/p^3\mathbb Z,
\label{eq:p3a}
\end{align}
where $H_{m}^{(\al)}$ denotes the higher harmonic number defined by
$H_{m}^{(\al)}=\sum_{n=1}^{m} \frac{1}{n^\al}$. By a rearrangement
analogous to the one in the proof of Lemma~\ref{lem:W1}, one sees
that $v_p(H_{p-1}^{(3)})\ge1$, whence we may disregard the last term
in the last line of \eqref{eq:p3a}. As it turns out, $H_{p-1}$ and
$H_{p-1}^{(2)}$ are directly related modulo $\mathbb Z/p^3\mathbb Z$.
Namely, we have
\begin{align} 
\notag
H_{p-1}&\equiv\sum _{\ep=1} ^{p-1}\ep^{-1}
\equiv 2^{-1}\sum _{\ep=1} ^{p-1}\left(\ep^{-1} + (p-\ep)^{-1}\right)
\equiv p2^{-1}\sum _{\ep=1} ^{p-1}\big(\ep(p-\ep)\big)^{-1}\\
\notag
&\equiv -p2^{-1}\sum _{\ep=1} ^{p-1}\ep^{-2}\big(1+p\ep^{-1}\big)
\equiv -p2^{-1}\sum _{\ep=1} ^{p-1}\big(\ep^{-2}+p\ep^{-3}\big)\\
&\equiv -p2^{-1}H_{p-1}^{(2)}-p^22^{-1}H_{p-1}^{(3)}\mod \mathbb Z/p^3\mathbb Z.
\label{eq:HnHn2}
\end{align}
Again, we may disregard the last term. If we substitute this
congruence in \eqref{eq:p3a}, then we obtain
$$
\sum _{r=1} ^{N}\sum _{\ep=1} ^{p-1}(rp-p+\ep)^{-1}
\equiv 
N^2H_{p-1}
\mod \mathbb Z/p^3\mathbb Z.
$$
Hence, the sum is congruent to zero 
modulo~$\mathbb Z/p^3\mathbb Z$ if and only
if $N$ is divisible by $p$
(recall Lemma~\ref{lem:W1} with $r=1$) or $v_p(H_{p-1})\ge3$,
that is, if $p$ is a Wolstenholme prime.
This establishes the assertion of the lemma.
\end{proof}

\section{The sequence $(t_N)_{N\ge1}$}
\label{sec:DworkKont}

The purpose of this section is to report on the evidence for our
conjecture that the largest integer $t_N$ such that
$q_{N,(N)}(z)^{1/t_N}\in\mathbb Z[[z]]$ is given by $t_N=\Xi_NN!$, 
that is, that Theorem~\ref{thm:3} with $k=1$ is optimal.
We assume $k=1$ throughout this section.

We first prove that there cannot be any prime number $p$ larger than
$N$ which divides $t_N$.

\begin{prop} \label{prop:p>N}
Let $p$ be a prime number and $N$ a positive integer with $p>N$.
Then there exists a positive integer $a<p$ such that
\begin{equation} \label{eq:p>N} 
v_p\big(B_{N}(a)H_{Na}\big)=0,
\end{equation}
where $B_N(m)=\frac {(Nm)!} {m!^N}$.
In particular, $p$ does not divide $t_N$. 
\end{prop}

\begin{proof}
If $N=1$, we choose $a=1$ to obtain $B_{1}(1)H_{1}=1$. On the other
hand, if $N>1$, we choose $a$ to be the least integer such that $aN\ge p$.
Since then $a<p$, we have
$$v_p\big(B_{N}(a)H_{Na}\big)=
\sum _{\ell = 1}^{\infty}\left(\fl{\frac {Na} {p^\ell}}-
N\fl{\frac {a} {p^\ell}}\right)
+v_p(H_{Na})=1-1=0.$$

To see that $p$ cannot divide $t_N$, we observe that
we have $C(a)=-pB_{N}(a)H_{Na}$ (for the sum
$C(\cdot)$ defined in \eqref{eq:C} in the case $k=1$) because $a<p$. 
Hence, the assertion \eqref{eq:p>N} can be
reformulated as $v_p\big(C(a)\big)=1$.
Since, because of $p>N$,
we have $v_p(N!)=0$ and $v_p(\Xi_N)=0$, it follows that 
$$C(a)\notin p^2\Xi_N N!\, \mathbb Z_p.$$
This means that 
one cannot increase the exponent of $p$ in \eqref{eq:Ccong} 
(with $k=1$) in our special case, and thus $p$ cannot divide $t_N$.
\end{proof}

So, if we hope to improve Theorem~\ref{thm:3} with $k=1$, then it must be by
increasing exponents of prime numbers $p\le N$ in \eqref{eq:Xi}. 
It can be checked directly that the exponent of $3$ cannot be
increased if $N=7$. (The reader should recall Remarks~\ref{rem:Xi7}(a)
in the Introduction.)
According to Remarks~\ref{rem:Xi7}(b), an
improvement is therefore only possible if $v_p(H_N)>2$ for some $p\le N$. 
Lemmas~\ref{lem:H_L}--\ref{lem:5} in Section~\ref{sec:aux} tell
that this does not happen with $p=2,3,5$, so that the exponents of
$2,3,5$ cannot be improved. 
(The same conclusion can also be drawn from \cite{boyd} for many
other prime numbers, but so far not for $83$, for example.)

We already discussed in Remarks~\ref{rem:Xi7}(c) 
whether there are any primes $p$ and integers $N$
such that $p\le N$ and $v_p(H_N)\ge3$. We recall that, so far, only five
examples are known, four of them involving $p=11$. 

The final result of this section shows that, if $v_p(H_N)=3$, the exponent
of $p$ in the definition of $\Xi_N$ in \eqref{eq:Xi} cannot be increased
so that Theorem~\ref{thm:3} would still hold. 
(The reader should recall Remarks~\ref{rem:Xi7}(b).)
\begin{prop} \label{prop:vp=3}
Let $p$ be a prime number and $N$ a positive integer with $p\le N$
and $v_p(H_N)=3$. If $p$ is not a Wolstenholme prime and $p$ does not
divide $N$, then\break
$q_{N,(N)}(z)^{1/p\Xi_N N!}\notin \mathbb Z[[z]]$.
\end{prop}

\begin{proof}
We assume that $p$ is not a Wolstenholme prime and that $p$ does not
divide $N$. In particular, this implies $\xi(p,N)=0$ and thus also
$v_p(\Xi_N)=2$.
By Lemmas~\ref{lem:H_L}--\ref{lem:5}, we can furthermore
assume that $p\ge7$.

Going back to the outline of the proof of Theorem~\ref{thm:3} in
Section~\ref{sec:2}, we claim that
\begin{equation} \label{eq:Cp}
C(p)=\big(B_N(1)H_N-B_N(p)pH_{Np}\big)\notin p^4 N!\,\mathbb{Z}_p,
\end{equation}
where $B_N(m)=\frac {(Nm)!} {m!^N}$.
(The claim here is the non-membership relation; the equality holds by
the definition of $C(\cdot)$ in \eqref{eq:Ccong}.)
This would imply that $C(p)\notin p^2 \Xi_NN!\,\mathbb Z_{p}$, 
and thus, by Lemma~\ref{lem:4}
(recall the outline of the proof of Theorem~\ref{thm:3} 
in Section~\ref{sec:2}), that
$q_{N,(N)}(z)^{1/p\Xi_N N!}\notin \mathbb Z[[z]]$, as desired.

To establish \eqref{eq:Cp}, we consider 
\begin{multline*}
H_N(B_N(1)-B_N(p))\\
=H_NN!\bigg(1-\frac 1{(p-1)!^N}\big(1\cdot 2\cdots (p-1)\big)
\big((p+1)\cdot (p+2)\cdots (2p-1)\big)\cdots\\
\times\cdots
\big((pN-p+1)\cdot (pN-p+2)\cdots (pN-1)\big) \bigg).
\end{multline*}
Using $v_{p}(H_N)=3$ and Wilson's theorem, we deduce 
\begin{equation} \label{eq:Zp} 
H_N(B_N(1)-B_N(p))\in p^4N!\,\mathbb Z_{p}.
\end{equation}
However, by Lemma~\ref{lem:congH} and the fact that 
$v_{p}\big(B_N(p)\big)=v_{p}\big(B_N(1)\big)=v_{p}(N!)$, we obtain
$$B_N (1)H_N-B_N(p)pH_{pN}\hbox{${}\not\equiv{}$}
H_N(B_N(1)-B_N(p))\mod p^4N!\,\mathbb Z_{p}.$$
Together with \eqref{eq:Zp}, this yields \eqref{eq:Cp}.
\end{proof}

To summarise the discussion of this section: if the conjecture in
Remarks~\ref{rem:Xi7}(c) that no prime $p$ and integer $N$ exists with
$v_p(H_N)\ge4$ should be true, then Theorem~\ref{thm:3} with $k=1$ is
sharp, that is, the sequence $(t_N)_{N\ge1}$ is given by
$t_N=\Xi_NN!$.

\section{Sketch of the proof of Theorem~\ref{thm:3a}}
\label{sec:Om}

In this section we discuss the proof of Theorem~\ref{thm:3a}.
Since it is completely analogous to the proof of Theorem~\ref{thm:3}
(see Section~\ref{sec:2}), we content ourselves with pointing out the
differences. At the end of the section, we present analogues of
Propositions~\ref{prop:p>N} and \ref{prop:vp=3}, addressing the
question of optimality of Theorem~\ref{thm:3a} with $k=1$.

First of all, by \eqref{eq:truemap}, we have
$$
\big(z^{-1}q_{\mathbf N}(z)\big)^{1/kN}=\exp(\widetilde G_{\mathbf
N}(z)/F_{\mathbf N}(z)),
$$
where $F_{\mathbf N}(z)$ is the series from the Introduction and
\begin{equation*}
\widetilde G_{\mathbf N}(z):=\sum_{m=1}^{\infty} \frac{(Nm)!^k}{m!^{kN}}
\big(H_{Nm}-H_m\big)\,z^m.
\end{equation*}

We must adapt the proof of Theorem~\ref{thm:3}, as outlined in
Section~\ref{sec:2}. Writing as before
$B_{\mathbf N}(m)=\frac{(Nm)!^k}{m!^{kN}}$, we must consider the sum
\begin{multline} \label{eq:Summe}
\widetilde C(a+Kp):=\sum_{j=0}^K B_{\mathbf N}(a+jp)B_{\mathbf N}(K-j)
\big((H_{N(K-j)}-H_{K-j})-p(H_{Na+Njp}-H_{a+jp})\big)  \\
=\sum_{j=0}^K B_{\mathbf N}(a+jp)B_{\mathbf N}(K-j)
(H_{N(K-j)}-pH_{Na+Njp}) \kern4cm\\-
\sum_{j=0}^K B_{\mathbf N}(a+jp)B_{\mathbf N}(K-j)
(H_{K-j}-pH_{a+jp}) 
\end{multline}
and show that it is in $p\Om_NN!^k\mathbb Z_p$ for all primes $p$, and for all 
non-negative integers $K$, $a$, and $j$ with $0\le a<p$. 
The special cases $K=a=0$, respectively $K=0$
and $a=1$, are equally simple to be handled directly here. 
We leave their verification to the reader and assume $a+Kp\ge2$ from now on.

If we now go through the outline of the proof of Theorem~\ref{thm:3} in
Section~\ref{sec:2}, we see that the crucial steps are the congruence
\eqref{eq:firstreduction}, and Lemmas~\ref{lem:12}--\ref{lem:11}.
The analogue
of \eqref{eq:firstreduction} in our context is the assertion that
\begin{multline}\label{eq:firstreductionII}
\widetilde C(a+Kp) \equiv \sum_{j=0}^K B_{\mathbf N}(a+jp)B_{\mathbf
N}(K-j) 
\big((H_{N(K-j)}-H_{K-j})-(H_{\lfloor
Na/p\rfloor+Nj}-H_j)\big)\\ \mod p\Om_N N!^k\mathbb{Z}_p
\end{multline}
for any
non-negative integers $m$, $K$, $a$, and $j$ with $0\le a<p$.
This follows easily from the congruence \eqref{eq:J} and
Lemma~\ref{lem:multinomial/N!} as long as $v_p(H_N-1)\le 0$. If
$v_p(H_N-1)>0$, then we use Lemma~\ref{lem:p<L2} below to conclude that
$N/p\ge4$, which together with Lemma~\ref{lem:multinomial/N!} and
Corollary~\ref{cor:C1} in Section~\ref{sec:aux} yields 
\begin{multline}\label{eq:firstreductionUW2}
\widetilde C(a+Kp) \equiv \sum_{j=0}^K B_{\mathbf N}(a+jp)B_{\mathbf N}(K-j) 
\big((H_{N(K-j)}-H_{K-j})-(H_{\lfloor
Na/p\rfloor+Nj}-H_j)\big)\\
\mod p^5 N!^k\mathbb{Z}_p
\end{multline}
as long as $a+Kp\ge2$. This is more than we actually need to establish
\eqref{eq:firstreductionII} in this case also.

The analogue of Lemma~\ref{lem:12} in our context is
\begin{equation} \label{eq:congrconj1U2}
B_{\mathbf N}(a+pj)\left(H_{Nj+\lfloor Na/p\rfloor} - H_{Nj}\right)
\in p \Om_N N!^k \mathbb{Z}_p\ ,
\end{equation}
which we must prove for any prime $p$, and non-negative integers $a$, $j$
with $0\le a<p$. We may use Lemma~\ref{lem:10} as it stands, but
instead of Lemma~\ref{lem:11} we should prove
\begin{equation} \label{eq:110U2}
B_{\mathbf N}(m)\big((H_{Nmp^s}-H_{mp^s})-
(H_{N\fl{m/p}p^{s+1}}-H_{\fl{m/p}p^{s+1}})\big)\in 
p^{-s}\Om_N N!^k \mathbb Z_p\ .
\end{equation}
Both \eqref{eq:congrconj1U2} and \eqref{eq:110U2} 
can be proved in a very similar way as we proved
Lemmas~\ref{lem:12} and \ref{lem:11} in Sections~\ref{sec:4} and
\ref{sec:6}, respectively. Indeed, as long as $v_p(H_N-1)\le 0$, there
is nothing to prove, since in this case, by the definitions of $\Om_N$
and $\Xi_N$, Lemma~\ref{lem:12} implies \eqref{eq:congrconj1U2}
immediately, as well as Lemma~\ref{lem:11} implies \eqref{eq:110U2}.

If, on the other hand, $v_p(H_N-1)>0$, then we need
substitutes for Lemmas~\ref{lem:H_L}--\ref{lem:p<L}, which were used
to accomplish the proofs of Lemmas~\ref{lem:12} and \ref{lem:11} in
the case that $v_p(H_N)>0$.

\begin{lem} \label{lem:H_L2}
For all positive integers $N\ge2$, we have $v_2(H_N-1)=-\fl{\log_2N}$.
\end{lem}

\begin{lem} \label{lem:32}
We have $v_3(H_{66}-1)=v_3(H_{68}-1)=1$.
For positive integers $N\notin\{1,66,68\}$, we have $v_3(H_N-1)\le0$.
\end{lem}

\begin{lem} \label{lem:52}
We have $v_5(H_3-1)=v_5(H_{21}-1)=v_5(H_{23}-1)=1$.
For positive integers $N\notin\{1,3,21,23\}$, we have $v_5(H_N-1)\le0$.
\end{lem}

\begin{lem} \label{lem:p<L2}
Let $p$ be a prime, and let $N$ be an integer with $N\ge2$ and $p\le N$.
Then the following assertions hold true:

\begin{enumerate} 
\item If $v_p(H_N-1)>0$ then $N\ge 4p$.
\item If $v_p(H_N-1)>0$ and $p\ne5$ then $N\ge 6p$.
\end{enumerate}
\end{lem}

These results can be proved in exactly the same way as 
Lemmas~\ref{lem:H_L}--\ref{lem:p<L}, respectively.
In comparison to Lemma~\ref{lem:p<L}, the statement of
Lemma~\ref{lem:p<L2} is in fact simpler, so that complications
that arose in Section~\ref{sec:2}
(such as \eqref{eq:firstreductionU3}, for example) do not arise here.

Second, we need a substitute for Lemma~\ref{lem:congH}.

\begin{lem} \label{lem:congH2}
For all primes $p\ge5$ and positive integers $N$, we have
\begin{equation} \label{eq:congH12}
p(H_{pN}-H_p)\equiv H_N-1\mod p^4\mathbb Z_p
\end{equation}
if and only if $p$ is a Wolstenholme prime or $N\equiv\pm1$~{\em mod}~$p$.
\end{lem}

\begin{proof}
From the proof of Lemma~\ref{lem:congH}, we know that 
$$
pH_{pN}-H_N\equiv pN^2H_{p-1}\mod \mathbb Z/p^4\mathbb Z.$$
As a consequence, we obtain
$$
p(H_{pN}-H_p)-(H_N-1)\equiv p(N^2-1)H_{p-1}\mod \mathbb Z/p^4\mathbb Z.$$
The assertion of the lemma follows now immediately.
\end{proof}

Finally, the computation for $m=1$ in Section~\ref{sec:6}
must be replaced by
\begin{align*}
B_{\mathbf N}(1)(H_{Np^s}-H_{p^s})
&=N!^k\Bigg(\frac {1} {p^s}(H_N-1)+
\frac {1} {p^{s-1}}
\underset{p\nmid\ep}{\sum _{\ep=p+1} ^{Np}}\frac {1} {\ep}+
\frac {1} {p^{s-2}}
\underset{p\nmid\ep}{\sum _{\ep=p^2+1} ^{Np^2}}\frac {1} {\ep}+
\underset{p^{s-2}\nmid \ep}{\sum _{\ep=p^s+1} ^{Np^s}}\frac {1}
{\ep}\Bigg),
\end{align*}
with the conclusion that
$$
B_{\mathbf N}(1)(H_{Np^s}-H_{p^s})\in p^{-s} N!^k\Om_N\mathbb Z_p
$$
being found in a completely analogous manner.

Altogether, this leads to a proof of Theorem~\ref{thm:3a}.

\section{The Dwork--Kontsevich sequence}
\label{sec:DK}

In this section, we address the question
of optimality of Theorem~\ref{thm:3a} when $k=1$, that is, whether,
given that $k=1$, the largest integer $u_N$ such that
$\big(z^{-1}q_{(N)}(z)\big)^{\frac{1}{Nu_N}} \in\mathbb{Z}[[z]]$ 
is given by $\Om_N N!$. 
Let us write $\widetilde q_{(N)}(z)$ for $\big(z^{-1}q_{(N)}(z)\big)^{1/N}$
with $k$ {\it being fixed to $1$.}
The first proposition shows
that there cannot be any prime number $p$ larger than
$N$ which divides $u_N$. We omit the proof since it is
completely analogous to the proof of Proposition~\ref{prop:p>N}
in Section~\ref{sec:DworkKont}.

\begin{prop} \label{prop:p>N2}
Let $p$ be a prime number and $N$ a positive integer with $p>N\ge2$.
Then there exists a positive integer $a<p$ such that
\begin{equation} \label{eq:p>N2} 
v_p\big(B_{N}(a)(H_{Na}-H_a)\big)=0.
\end{equation}
In particular, $p$ does not divide $u_N$. 
\end{prop}

So, if we hope to improve Theorem~\ref{thm:3a} with $k=1$, then it must be by
increasing exponents of prime numbers $p\le N$ in \eqref{eq:Om}. 
According to Remarks~\ref{rem:Om}(b) in the Introduction, an
improvement is therefore only possible if 
$v_p(H_N-1)>2$ for some $p\le N$. 

The final result of this section shows that, if $v_p(H_N-1)=3$
(for which, however, so far no examples are known; see
Remarks~\ref{rem:Om}(c)), the exponent
of $p$ in the definition of $\Om_N$ in \eqref{eq:Om} cannot be increased
so that Theorem~\ref{thm:3a} with $k=1$ would still hold. 

\begin{prop} \label{prop:vp=32}
Let $p$ be a prime number and $N$ a positive integer with $p\le N$
and $v_p(H_N)=3$. If $p$ is not a Wolstenholme prime and 
$N\hbox{${}\not\equiv{}$}\pm1$~{\em mod}~$p$, then
$\widetilde q_{(N)}(z)^{\frac{1}{p\Om_{N}N!}}
\notin\mathbb{Z}[[z]]$.
\end{prop}

Again, the proof is completely analogous to the proof of 
Proposition~\ref{prop:vp=3} in Section~\ref{sec:DworkKont}, 
which we therefore omit.

So, if the conjecture in
Remarks~\ref{rem:Om}(c) that no prime $p$ and integer $N$ exists
with\break
$v_p(H_N-1)\ge4$ should be true, then Theorem~\ref{thm:3a} with $k=1$ is
optimal, that is, the Dwork--Kontsevich sequence $(u_N)_{N\ge1}$ is given by
$u_N=\Om_NN!$.

\section*{Acknowledgements}
The authors are extremely grateful to Alessio Corti 
and Catriona Maclean for illuminating discussions 
concerning the geometric side of our work, 
and to David Boyd for helpful information on
the $p$-adic behaviour of the harmonic numbers $H_N$ 
and for communicating to us the value \eqref{eq:boyd} from his files
from 1994. They also thank the referees for 
an extremely careful reading of the original manuscript.

\def\refname{Bibliography}

\end{document}